\documentclass[12pt,a4paper]{article}

\usepackage{amsthm,enumerate,graphicx}
\usepackage{amsmath,latexsym,amssymb,amsfonts}
\usepackage{eucal}
\usepackage{eufrak}
\usepackage{comment}

\newcounter{Scounter}
\newtheorem{thm}{Theorem}
\newtheorem{refthm}[thm]{Theorem}

\newtheorem{lem}{Lemma}

\newtheorem{Claim}{Claim}

\newtheorem{subclaim}{Subclaim}

\newcommand{\ir}[3]{#1(#2,#3)}
\newcommand{\IR}[3]{#1[#2,#3]}
\newcommand{\Ir}[3]{#1[#2,#3)}
\newcommand{\iR}[3]{#1(#2,#3]}
\newcommand{\IL}[3]{\ola{#1}[#2,#3]}

\newcommand{\sm}{\setminus}
\newcommand{\ms}{-}

\newcommand{\bs}{\setminus}

\newcommand{\dg}[2]{d_{#1}(#2)}

\newcommand{\ola}{\overleftarrow}

\newcounter{Ccounter}
\bfseries\normalfont
\begin{document}

\title{A degree sum condition on the order, 
the connectivity and the independence number for Hamiltonicity}

\author{%
Shuya Chiba$^{1}$%
\thanks{Supported by JSPS KAKENHI Grant Number 17K05347.}
\and
Michitaka Furuya$^{2}$%
\thanks{Supported by JSPS KAKENHI Grant Number 26800086.}
\and
Kenta Ozeki$^{3}$%
\thanks{This work was supported by JST ERATO Kawarabayashi Large Graph Project, Grant Number JPMJER1201, Japan.}
\and
Masao Tsugaki$^{4}$%
\and 
Tomoki Yamashita$^{5}$%
\thanks{Supported by JSPS KAKENHI Grant Number 16K05262.} \\
\small
$^1$\small\textsl{Applied Mathematics, Faculty of Advanced Science and Technology, Kumamoto University,}\vspace{-6pt}\\ 
\small\textsl{2-39-1 Kurokami, Kumamoto 860-8555, Japan.}\vspace{-6pt}\\
\small{Email address: \texttt{schiba@kumamoto-u.ac.jp}}\vspace{3pt}\\
\small
$^{2}$\small\textsl{College of Liberal Arts and Sciences, Kitasato University,}\vspace{-6pt}\\
\small\textsl{1-15-1 Kitasato, Minami-ku, Sagamihara, Kanagawa 252-0373, Japan.}\vspace{-6pt}\\
\small Email address: \texttt{michitaka.furuya@gmail.com}\vspace{3pt}\\
\small
$^{3}$\small\textsl{Faculty of Environment and Information Sciences, Yokohama National University,}\vspace{-6pt}\\
\small\textsl{79-7 Tokiwadai, Hodogaya-ku, Yokohama 240-8501, Japan.}\vspace{-6pt}\\
\small{Email address: \texttt{ozeki-kenta-xr@ynu.ac.jp}}\vspace{3pt}\\
\small
$^{4}$\small\textsl{Department of Applied Mathematics, Tokyo University of Science,}\vspace{-6pt}\\
\small\textsl{1-3 Kagurazaka, Shinjuku-ku, Tokyo 162-8601, Japan.}\vspace{-6pt}\\
\small{Email address: \texttt{tsugaki@hotmail.com}} \vspace{3pt}\\
\small
$^{5}$\small\textsl{Department of Mathematics, Kindai University,}\vspace{-6pt}\\
\small\textsl{3-4-1 Kowakae, Higashi-Osaka, Osaka 577-8502, Japan. }\vspace{-6pt}\\
\small{Email address: \texttt{yamashita@math.kindai.ac.jp}}}
\date{}
\maketitle
\begin{abstract}
In [Graphs Combin.~24 (2008) 469--483.],
the third author and the fifth author
conjectured that
if $G$ is a $k$-connected graph
such that
$\sigma_{k+1}(G) \ge |V(G)|+\kappa(G)+(k-2)(\alpha(G)-1)$,
then $G$ contains a Hamiltonian cycle,
where $\sigma_{k+1}(G)$, $\kappa(G)$ and $\alpha(G)$
are
the minimum degree sum of $k+1$ independent vertices,
the connectivity and the independence number of $G$, respectively.
In this paper,
we settle this conjecture. 
This is an improvement
of the result obtained by Li:
If $G$ is a $k$-connected graph
such that 
$\sigma_{k+1}(G) \ge |V(G)|+(k-1)(\alpha(G)-1)$,
then $G$ is Hamiltonian.
The degree sum condition is best possible.
\end{abstract}

\section{Introduction}

\subsection{Degree sum condition for graphs with high connectivity to be Hamiltonian}
\label{intro}

In this paper,
we consider only finite undirected graphs
without loops or multiple edges.
For standard graph-theoretic terminology not explained,
we refer the reader to \cite{Bondybook}.

A \textit{Hamiltonian cycle} of a graph is a cycle containing all the vertices of the graph.
A graph having a Hamiltonian cycle is called a \textit{Hamiltonian graph}.
The Hamiltonian problem has long been fundamental in graph theory.
Since it is NP-complete,
no easily verifiable necessary and sufficient condition seems to exist.
Then instead of that,
many researchers have investigated sufficient conditions
for a graph to be Hamiltonian.
In this paper, we deal with a degree sum type condition,
which is one of the main stream of this study.

We introduce four invariants, including degree sum, 
which play important roles for the existence of a Hamiltonian cycle.
Let $G$ be a graph.
The number of vertices of $G$
is called its \textit{order},
denoted by $n(G)$.
A set $X$ of vertices in $G$ is called \textit{an independent set in $G$}
if no two vertices of $X$ are adjacent in $G$.
The \textit{independence number} of $G$
is defined by
the maximum cardinality of an independent set in $G$,
denoted by $\alpha(G)$.
For two distinct vertices $x,y \in V(G)$,
the \textit{local connectivity} $\kappa_G(x,y)$
is defined to be the maximum number of internally-disjoint paths
connecting $x$ and $y$ in $G$.
A graph $G$ is \textit{$k$-connected}
if
$\kappa_G(x,y) \ge k$
for any two distinct vertices $x, y \in V(G)$.
The \textit{connectivity} $\kappa(G)$ of $G$
is the maximum value of $k$ for which $G$ is $k$-connected.
We denote by $N_{G}(x)$ and $d_{G}(x)$ 
the neighbor and the degree of a vertex $x$ in $G$, respectively. 
If $\alpha(G) \ge k$,
let
$$
\sigma _{k} (G) = 
\min \big\{\sum_{x\in X}\dg{G}{x}
\colon \text{$X$ is an independent set in $G$ 
with $|X|=k$} \big\};
$$
otherwise
let $\sigma _{k} (G) = 
+\infty$.
If the graph $G$ is clear from the context,
we simply
write
$n$,
$\alpha$, $\kappa$ and
$\sigma_k$ instead of
$n(G)$,
$\alpha(G)$,
$\kappa(G)$
and $\sigma_k(G)$, respectively.


\paragraph{}
One of the main streams of the study 
of the Hamiltonian problem
is,
as mentioned above,
to consider 
degree sum type
sufficient conditions
for graphs to have a Hamiltonian cycle.
We list some of them below.
(Each of the conditions is best possible in some sense.)

\begin{refthm}
\label{degresult}
Let $G$ be a graph of order at least three.
If $G$ satisfies one of the following,
then $G$ is Hamiltonian.
\begin{enumerate}[{\upshape (i)}]
\item
{\upshape (Dirac \cite{Dirac})}
The minimum degree of $G$ is at least $\frac{n}{2}$.
\item
{\upshape (Ore \cite{Ore})}
$\sigma_2 \ge n$.
\item
{\upshape (Chv\'{a}tal and Erd\H{o}s \cite{Chvatal&Erdos})}
$\alpha \le \kappa$.
\item
{\upshape (Bondy \cite{Bondy})}
$G$ is $k$-connected and 
$\displaystyle
\sigma_{k+1} > \frac{(k+1)(n-1)}{2}$.
\item
{\upshape (Bauer, Broersma, Veldman and Li \cite{BBVL})}
$G$ is $2$-connected and $\sigma_3 \ge n + \kappa$.
\end{enumerate}
\end{refthm}

To be exact,
Theorem \ref{degresult} (iii) is not 
a degree sum type condition,
but it is closely related.
Bondy \cite{Bondyrem} showed that 
Theorem \ref{degresult} (iii) implies (ii).
The current research of this area is based on
Theorem \ref{degresult} (iii).
Let us explain how to expand the research from Theorem \ref{degresult} (iii):
Let $G$ be a $k$-connected graph,
and suppose that one wants to consider whether $G$ is Hamiltonian.
If $\alpha \leq k$,
then
it follows from Theorem \ref{degresult} (iii) that 
$G$ is Hamiltonian.
Hence 
we may assume that $\alpha \geq k+1$,
that is,
$G$ has an independent set of order $k+1$.
Thus,
it is natural to consider a $\sigma_{k+1}$ condition
for a $k$-connected graph.
Bondy \cite{Bondy}
gave a $\sigma_{k+1}$ condition of Theorem \ref{degresult} (iv).

\paragraph{}
In this paper,
we give a much weaker 
$\sigma_{k+1}$ condition
than that of Theorem \ref{degresult} (iv).

\begin{thm}\label{main}
Let $k$ be an integer with $k \ge 1$
and
let $G$ be a $k$-connected graph.
If $$\sigma_{k+1} \ge n+\kappa+(k-2)(\alpha-1),$$
then $G$ is Hamiltonian.
\end{thm}

Theorem \ref{main}
was conjectured by
Ozeki and Yamashita \cite{OY},
and 
has been proven
for small integers $k$:
The case $k=2$ of Theorem \ref{main} 
coincides Theorem \ref{degresult} (v).
The cases $k=1$ and $k=3$ were shown 
by Fraisse and Jung \cite{FJ},
and by Ozeki and Yamashita \cite{OY}, respectively.

\subsection{Best possibility of Theorem \ref{main}}
\label{degbest}

In this section,
we show that 
the $\sigma_{k+1}$ condition in Theorem \ref{main} is 
best possible in some senses.

We first 
discuss the lower bound 
of the $\sigma_{k+1}$ condition.
For an integer $l\geq 2$ and $l$ vertex-disjoint graphs $H_{1},\ldots ,H_{l}$, 
we define the graph $H_{1}+\cdots +H_{l}$ from the union of $H_{1},\ldots ,H_{l}$ 
by joining every vertex of $H_{i}$ to every vertex of $H_{i+1}$ for $1\leq i\leq l-1$.
Fix an integer $k\geq 1$.
Let $\kappa $, $m$ and $n$ be integers
with $k\leq \kappa <m$ and $2m+1\leq n\leq 3m-\kappa $.
Let $G_{1}=K_{n-2m}+\overline{K}_{\kappa}+\overline{K}_{m}+\overline{K}_{m-\kappa}$,
where $K_{l}$ denotes a complete graph of order $l$
and
$\overline{K}_{l}$ denotes the complement of $K_{l}$.
Then
$\alpha(G_{1})=m+1$, $\kappa (G_{1})=\kappa $ and
\begin{eqnarray*}
\sigma_{k+1}(G_{1})
&=& (n - 2m -1 + \kappa) + km\\
&=&
n(G_{1}) + \kappa(G_{1}) +(k-2)(\alpha(G_{1})-1)-1.
\end{eqnarray*}
(Note that it follows from condition
``$n\leq 3m-\kappa $''
that $n - 2m -1 + \kappa < m$.)
Since deleting all the vertices in $\overline{K}_{\kappa}$
and those in $\overline{K}_{m -\kappa}$
breaks $G_1$ into $m+1$ components,
we see that $G_{1}$ has no Hamiltonian cycle.
Therefore,
the $\sigma_{k+1}$ condition in Theorem \ref{main} is 
best possible.

We next
discuss 
the relation between the coefficient of $\kappa$
and that of $\alpha - 1$.
By Theorem \ref{degresult} (iii),
we may assume that $\alpha \ge \kappa+1$.
This implies that 
$$n+\kappa+(k-2)(\alpha-1) 
\ge n+ (1 + \varepsilon) \kappa + (k-2-\varepsilon)(\alpha-1)$$
for arbitrarily $\varepsilon > 0$.
Then
one may expect
that 
the $\sigma_{k+1}$ condition in Theorem \ref{main}
can be replaced with
``$n+ (1 + \varepsilon) \kappa + (k-2-\varepsilon)(\alpha-1)$''
for some $\varepsilon > 0$.
However, the graph $G_{1}$ as defined above shows that it is not true:
For any $\varepsilon > 0$,
there exist two integers $m$ and $\kappa$
such that 
$\varepsilon(m-\kappa) \geq 1$.
If we construct the  above graph $G_1$
from such integers $m$ and $\kappa$,
then
we have 
\begin{eqnarray*}
\sigma_{k+1}(G_{1})
&=& n + \kappa + (k-2) m - 1\\
&=&
n + (1 + \varepsilon) \kappa
+ (k-2-\varepsilon)m-1
+\varepsilon(m- \kappa)\\
&\ge&
n(G_{1}) + (1 + \varepsilon) \kappa(G_{1})
+ (k-2-\varepsilon)\big(\alpha(G_{1}) -1\big),
\end{eqnarray*}
but $G_{1}$ is not Hamiltonian.
This means that 
the coefficient $1$ of $\kappa$
and 
the coefficient $k-2$ of $\alpha - 1$
are,
in a sense,
best possible.

\subsection{Comparing Theorem \ref{main} to other results}
\label{compairothers}

In this section,
we 
compare
Theorem \ref{main} 
to 
Theorem \ref{degresult} (iv)
and
Ota's result (Theorem \ref{Ota}).

We first show that
the $\sigma_{k+1}$ condition of Theorem \ref{main} is  weaker than that of Theorem \ref{degresult} (iv).
Let $G$ be a $k$-connected graph
satisfying the $\sigma_{k+1}$ condition of Theorem \ref{degresult} (iv).
Assume that
$\alpha \ge (n+1)/2$.
Let $X$ be an independent set of order at least $(n+1)/2$.
Then
$|V(G) \setminus X| \le (n-1)/2$
and
$|V(G) \setminus X| \ge k$
since $V(G) \setminus X$ is a cut set.
Hence $(n+1)/2 \ge k+1$,
and
we can take a subset $Y$ of $X$ with $|Y|=k+1$.
Then
$N_G(y) \subseteq V(G) \sm X$ for $y\in Y$,
and hence
$\sum_{y \in Y} d_G(y) 
\le  (k+1)|V(G) \bs X|
\le 
(k+1)(n-1)/2$.
This contradicts 
the $\sigma_{k+1}$ condition
of Theorem \ref{degresult} (iv).
Therefore
$n/2 \ge \alpha$.
Moreover,
by Theorem \ref{degresult} (iii),
we may assume that
$\alpha \geq \kappa+1$.
Therefore,
the following inequality holds:
\begin{eqnarray*}
\sigma _{k+1}
&>&\frac{(k+1)(n-1)}{2}\\
&=&n-1+\frac{(k-1)(n-1)}{2}\\
&\ge&n-1+\frac{(k-1)(2\alpha-1)}{2}\\
&\ge&n-1+(k-1)(\alpha-1)\\
&\ge& n+\kappa+(k-2)(\alpha-1)-1.
\end{eqnarray*}
Thus,
the $\sigma_{k+1}$ condition
of Theorem \ref{degresult} (iv)
implies 
that
of Theorem \ref{main}.

We next compare
Theorem \ref{main}
to the following Ota's result.

\begin{thm}[Ota \cite{Ota}]\label{Ota}
Let $G$ be a $2$-connected graph.
If $\sigma_{l+1} \ge n+l(l-1)$
for all integers $l$
with  $l \ge \kappa$,
then $G$ is Hamiltonian.
\end{thm}

We first mention about the reason
to compare Theorem \ref{main} to Theorem \ref{Ota}. 
Li \cite{Hao13}
proved the following theorem,
which was conjectured by 
Li, Tian, and Xu \cite{LTX10}.
(Harkat-Benhamadine, Li and Tian \cite{HLT},
and
Li, Tian, and Xu \cite{LTX10}
have already proven the case $k=3$ and the case $k=4$,
respectively.)

\begin{thm}[Li \cite{Hao13}]
\label{coro}
Let $k$ be an integer with $k \ge 1$
and
let $G$ be a $k$-connected graph.
If $\sigma_{k+1} \ge n+(k-1)(\alpha-1)$,
then $G$ is Hamiltonian.
\end{thm}

In fact,
Li showed Theorem \ref{coro}
just as a corollary of Theorem \ref{Ota}.
Note that
Theorem \ref{main} is,
assuming Theorem \ref{degresult} (iii),
an improvement of Theorem \ref{coro}.
Therefore
we should show
that 
Theorem \ref{main} cannot be implied 
by Theorem \ref{Ota}.
(Ozeki, in his Doctoral Thesis \cite{Ozeki}, compared the relation between 
several theorems,
including 
Theorem \ref{degresult} (i), (ii), (iii) and (v),
the case $k=3$ of
Theorems 
\ref{main} and \ref{coro},
and Theorem \ref{Ota}.)

Let $\kappa, r,k,m$ be integers such that
$4 \le r$,
$3 \le k \le \kappa-2$
and
$m = (k+1)(r-2)+4.$
Let $G_{2}=K_{1} + \overline{K}_{\kappa}+K_{\kappa+m-r}+(\overline{K}_{m}+K_{r})$.
Then
$n(G_{2})=2\kappa + 2m + 1$,
$\kappa(G_{2})=\kappa$
and
$\alpha(G_{2})=\kappa + m$.
Since 
\begin{eqnarray*}
\kappa + k(\kappa + m) - (k+1) (\kappa + m -r +1)
&=& (k+1)(r-1) -m \\
&=& (k+1)(r-1) - (k+1)(r-2)-4 \\
&=& k-3 \\
&\ge&  0,
\end{eqnarray*}
it follows that 
\begin{eqnarray*}
\sigma_{k+1}(G_2) 
&=& \min\big\{\kappa + k(\kappa + m),\ (k+1) (\kappa + m -r+1)\big\} \\
&=& \kappa+k(\kappa+m) -(k-3)\\
&=& (2\kappa + 2m + 1) + \kappa + (k-2)(\kappa + m - 1)\\
&=& n(G_2) + \kappa(G_2) + (k-2)(\alpha(G_2) - 1).
\end{eqnarray*}
Hence the assumption of Theorem \ref{main} holds.
On the other hand,
for $l=\alpha(G_2) - 1 = \kappa+m-1$,
we have 
\begin{eqnarray*}
n(G_2) + l(l-1) - \sigma_{l+1}(G_2)
&=& (2\kappa + 2m+1) + (\kappa+m-1)(\kappa + m-2) \\
&& - \big\{\kappa(\kappa+m-r+1)+m(\kappa+m)\big\} \\
&=& \kappa (r - 2) -m +3
\\
&=& \kappa (r - 2)- (k+1)(r-2)-4 +3
\\
&=& (\kappa-k-1)(r - 2)-1
\\
&\ge& (r - 2)-1 
\\
&>& 0.
\end{eqnarray*}
Hence the assumption of Theorem \ref{Ota}
does not hold.
These yield that 
for the graph $G_2$,
we can apply Theorem \ref{main},
but cannot apply Theorem \ref{Ota}.

\section{Notation and lemmas}


Let $G$ be a graph
and $H$ be a subgraph of $G$,
and
let $x \in V(G)$ and $X \subseteq V(G)$.
We denote by 
$N_G(X)$
the set of vertices in $V(G) \sm X$
which are adjacent to some vertex in $X$. 
We define
$N_H(x) = N_G(x) \cap V(H)$
and
$\dg{H}{x} = |N_H(x)|$.
Furthermore,
we define
$N_H(X) = N_G(X) \cap V(H)$.
If there is no fear of confusion,
we often identify $H$
with its vertex set $V(H)$.
For example,
we often write $G\ms H$
instead of $G\ms V(H)$. 
For a subgraph $H$,
a path $P$ is called an \textit{$H$-path} 
if both end vertices of $P$ are contained in $H$ 
and all internal vertices are not contained in $H$.
Note that each edge of $H$ is an $H$-path.

Let $C$ be a cycle (or a path) with a fixed orientation in a graph $G$. 
For $x,y \in V(C)$,
we denote
by $\IR{C}{x}{y}$
the path from $x$ to $y$
along the orientation of ${C}$.
The reverse sequence
of $\IR{C}{x}{y}$
is denoted by $\IL{C}{y}{x}$.
We denote $\IR{C}{x}{y} \ms \{x,y\}$,
$\IR{C}{x}{y} \ms \{x\}$
and $\IR{C}{x}{y} \ms \{y\}$
by $\ir{C}{x}{y}$, $\iR{C}{x}{y}$ and $\Ir{C}{x}{y}$, respectively.
For $x \in V(C)$, 
we denote the successor and the predecessor 
of $x$ on $C$ by $x^{+}$ and $x^{-}$, respectively.
For $X \subseteq V(C)$,
we define $X^{+} = \{x^{+} : x \in X\}$ 
and $X^{-} = \{x^{-} : x \in X \}$. 
Throughout this paper, we consider that every cycle
has a fixed orientation.

In this paper, we extend the concept of 
\textit{insertible},
introduced by Ainouche \cite{Ainouche92},
which has been used for the proofs of the results on cycles.

Let $G$ be a graph, 
and $H$ be a subgraph of $G$.
Let 
$X(H)= \{u \in V(G-H) : \mbox{$uv_1, uv_2 \in E(G)$ for some $v_1v_2 \in E(H)$}  \}$, 
let 
$I(x;H) = \{v_1v_2 \in E(H) : xv_1, xv_2 \in E(G)\}$ for $x \in V(G-H)$,
and let
$Y(H)= \{u \in V(G-H) : \text{ $d_{H}(u)  \ge \alpha(G)$}\}$.

\begin{lem}
\label{D cup Q is hamilton}
Let $D$ be a cycle of a graph $G$. 
Let $k$ be a positive integer
and
let 
$Q_{1}, Q_{2},\ldots, Q_{k}$ be paths of $G - D$ with fixed orientations  
such that $V(Q_{i}) \cap V(Q_{j}) = \emptyset$ for $1 \le i<j \le k$. 
If the following {\rm(I)} and {\rm(II)} hold, 
then $G[V(D \cup Q_{1} \cup Q_{2}\cup \cdots \cup Q_{k})]$ is Hamiltonian.
\begin{enumerate}[{\upshape(I)}]
\item 
For $1 \le i  \le k$ and $a \in V(Q_{i})$,
$a \in X(D)\cup Y(\iR{Q_{i}}{a}{b_{i}} \cup D)$, 
where $b_{i}$ is the last vertex of $Q_{i}$. 
\item 
For $1 \le i<j \le k$, $x \in V(Q_{i})$ and $y \in V(Q_{j})$,
$I(x; D) \cap I(y; D) = \emptyset$.
\end{enumerate}
\end{lem}

\begin{proof}
We can easily see that 
$G[V(D \cup Q_{1} \cup Q_{2}\cup \cdots \cup Q_{k})]$ contains a cycle $D^{*}$ 
such that $V(D) \cup \big( X(D) \cap V(Q_1 \cup Q_2 \cup \cdots \cup Q_{k}) \big) \subseteq V(D^{*})$. 
In fact,
we can insert 
all vertices of $X(D) \cap V(Q_{1})$ into $D$
by choosing the following 
$u_{1}, v_{1} \in V(Q_{1})$ and 
$w_{1}w_{1}^{+} \in E(D)$ inductively.
Take the first vertex $u_{1}$ in $X(D) \cap V(Q_{1})$ 
along the orientation of $Q_{1}$,
and let $v_{1}$ be the last vertex in $X(D) \cap V(Q_{1})$ on $Q_{1}$ 
such that $I(u_{1}; D) \cap I(v_{1}; D) \neq \emptyset$.
Then we can insert 
all vertices of $\IR{Q_{1}}{u_{1}}{v_{1}}$
into $D$. 
To be exact,
taking $w_{1}w_{1}^{+} \in I(u_{1}; D) \cap I(v_{1}; D)$, 
$D_{1}^{1}:=w_{1}\IR{Q_{1}}{u_{1}}{v_{1}}\IR{D}{w_{1}^{+}}{w_{1}}$ 
is such a cycle.
By the choice of $u_{1}$ and $v_{1}$, 
$w_{1}w_{1}^{+} \notin I(x; D)$ for all $x \in V(Q_{1} - \IR{Q_{1}}{u_{1}}{v_{1}})$, 
and 
$X(D) \cap V(Q_{1} - \IR{Q_{1}}{u_{1}}{v_{1}})$
is contained in some component of $Q_{1} - \IR{Q_{1}}{u_{1}}{v_{1}}$.
Moreover,
note that $E(D) \sm \{w_{1}w_{1}^{+}\} \subseteq E(D_{1}^{1})$.
Hence by repeating this argument, 
we can obtain a cycle $D_{1}^{*}$ of $G[V(D \cup Q_1)]$
such that $V(D) \cup \big( X(D) \cap V(Q_{1}) \big) \subseteq V(D_{1}^{*})$ 
and $E(D) \setminus \bigcup_{x \in V(Q_{1})}I(x; D) \subseteq E(D_{1}^{*})$.
Then by (II),
$I(x;D) \subseteq E(D_{1}^{*})$ for all $x \in V(Q_{2} \cup \cdots \cup Q_{k})$.
Therefore
$G[V(D \cup Q_{1} \cup Q_{2}\cup \cdots \cup Q_{k})]$ contains a cycle $D^{*}$ 
such that 
$V(D) \cup \big( X(D) \cap V(Q_1 \cup Q_2 \cup \cdots \cup Q_{k}) \big) \subseteq V(D^{*})$. 

We choose a cycle $C$ of $G[V(D \cup Q_{1} \cup Q_{2}\cup \cdots \cup Q_{k})]$ 
containing all vertices in $V(D) \cup \big( X(D) \cap V(Q_1 \cup Q_2 \cup \cdots \cup Q_{k}) \big)$
so that $|C|$ is as large as possible.
Now, we change the ``base'' cycle from $D$ to $C$,
and use the symbol ${(\cdot)}^{+}$ for the orientation of $C$.
Suppose that 
$V(Q_{i} \ms C) \neq \emptyset$ for some $i$ with $i \in \{1, 2,\ldots,k\}$. 
We may assume that $i = 1$.
Let $w$ be the last vertex in $V(Q_{1} \ms C)$ along $Q_{1}$. 
Since $C$ contains all vertices in $X(D) \cap V(Q_{1})$,
it follows from (I) that $w \in Y(Q_{1}(w,b_1] \cup D)$,
that is,
$|N_{G}(w) \cap V(Q_{1}(w,b_1] \cup D)| \ge \alpha(G)$. 
By the choice of $w$,
we obtain
$V(Q_{1}(w,b_1] \cup D) \subseteq V(C) $. 
Therefore
$|N_{C}(w)^{+} \cup \{w\}| \ge |N_{G}(w) \cap V(Q_{1}(w,b_1] \cup D)| + 1 \ge \alpha(G) + 1$. 
This implies that 
$N_{C}(w)^{+} \cup \{w\}$ is not an independent set in $G$. 
Hence
$wz^{+} \in E(G)$ for some $z \in N_{C}(w)$ 
or
$z_{1}^{+}z_{2}^{+} \in E(G)$ for some distinct $z_{1}, z_{2} \in N_{C}(w)$. 
In the former case, let 
$C' = w\IR{C}{z^{+}}{z}w$, 
and in the latter case, 
let $C' = w\IR{\ola{C}}{z_{1}}{z_{2}^{+}}\IR{C}{z_{1}^{+}}{z_{2}}w$. 
Then $C'$ 
is a cycle of $G[V(D \cup Q_{1} \cup Q_{2}\cup \cdots \cup Q_{k})]$ such that 
$V(C) \cup \{w\} \subseteq V(C')$, which contradicts the choice of $C$. 
Thus $V(Q_{1} \cup Q_{2}\cup \cdots \cup Q_{k})$ are contained in $C$, 
and  hence $C$ is a Hamiltonian cycle of 
$G[V(D \cup Q_{1} \cup Q_{2}\cup \cdots \cup Q_{k})]$. 
\end{proof}

In the rest of this section, 
we fixed the following notation. 
Let $C$ be a longest cycle in a graph $G$, 
and $H_{0}$ be a component of $G \ms C$. 
For $u \in N_{C}(H_{0})$, 
let $u' \in N_{C}(H_{0})$ be a vertex 
such that $\ir{C}{u}{u'}  \cap N_{C}(H_{0})=\emptyset$,
that is,
$u'$ is the successor of $u$ in $N_C(H_{0})$ along the orientation of $C$.

For $u \in N_{C}(H_{0})$, 
a vertex $v \in \ir{C}{u}{u'}$ 
is \textit{insertible}
if $v \in X(\IR{C}{u'}{u}) \cup Y(\iR{C}{v}{u})$.
A vertex in $\ir{C}{u}{u'}$ is said to be \textit{non-insertible}
if it is not insertible.

\begin{lem}
\label{non-insertible}
There exists a non-insertible vertex
in $\ir{C}{u}{u'}$ for $u \in N_{C}(H_{0})$. 
\end{lem}

\begin{proof}
Let $u \in N_{C}(H_{0})$, 
and suppose that 
every vertex in $\ir{C}{u}{u'}$ is insertible. 
Let 
$P$ be a $C$-path joining $u$ and $u'$ with 
$V(P) \cap V(H_{0}) \neq \emptyset$. 
Let
$D = \IR{C}{u'}{u}P[u,u']$ 
and $Q = \ir{C}{u}{u'}$.
Let $v \in V(Q)$.
Since $v$ is insertible,
it follows that
$v \in X(\IR{C}{u'}{u}) \cup Y(C(v,u])$.
Since $\IR{C}{u'}{u}$ is a subpath of $D$,
we have
$v \in X(D) \cup Y(Q(v,u') \cup D)$.
Hence, by Lemma \ref{D cup Q is hamilton}, 
$G[V(D \cup Q)]$ is Hamiltonian, 
which contradicts the maximality of $C$.  
\end{proof}

\begin{figure}[h]
\begin{center}
\includegraphics{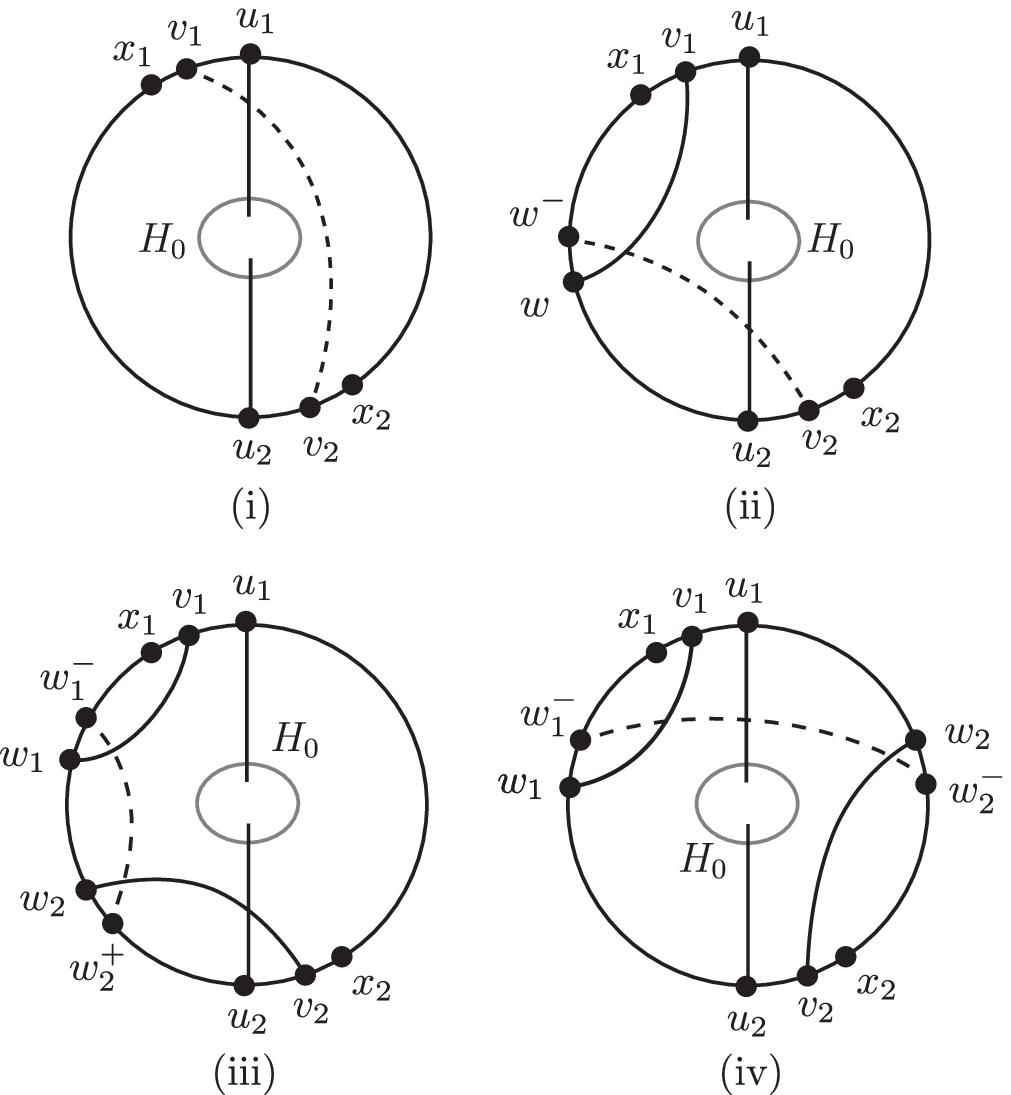}
\caption{Lemma \ref{insertible}}
\label{insertible-pict2}
\end{center}
\end{figure}

\begin{lem}
\label{insertible}
Let $u_{1}, u_{2} \in N_{C}(H_{0})$ with $u_{1} \neq u_{2}$, 
and let $x_{i}$ be the first non-insertible vertex along $\ir{C}{u_i}{u_i'}$ for $i \in \{ 1, 2\}$.  
Then the following hold (see Figure \ref{insertible-pict2}).
\begin{enumerate}[{\upshape (i)}]

\item\label{crossing}
There exists no $C$-path joining $v_{1} \in \iR{C}{u_{1}}{x_{1}}$ and $v_{2} \in \iR{C}{u_{2}}{x_{2}}$.
In particular,
$x_{1}x_{2} \not\in E(G)$.  

\item\label{crossing2}
If there exists a $C$-path joining $v_{1} \in \iR{C}{u_{1}}{x_{1}}$ 
and $w \in \iR{C}{v_{1}}{u_2}$, 
then there exists no $C$-path joining $v_{2} \in \iR{C}{u_{2}}{x_{2}}$ and $w^{-}$. 

\item\label{W1}
If there exist a $C$-path joining $v_{1} \in \iR{C}{u_{1}}{x_{1}}$ and $w_{1} \in \ir{C}{v_{1}}{u_{2}}$ 
and a $C$-path joining $v_{2} \in \iR{C}{u_{2}}{x_{2}}$ and $w_{2} \in \Ir{C}{w_{1}}{u_{2}}$, 
then there exists no $C$-path joining $w_{1}^{-}$ and $w_{2}^{+}$. 

\item\label{W2}
If for each $i \in \{1, 2\}$, 
there exists a $C$-path joining $v_{i} \in \iR{C}{u_{i}}{x_{i}}$ and $w_{i} \in \iR{C}{v_{i}}{u_{3-i}}$, 
then there exists no $C$-path joining $w_{1}^{-}$ and $w_{2}^{-}$. 

\end{enumerate}
\end{lem}

\begin{proof} 
Let $P_{0}$ be a $C$-path which connects $u_1$ and $u_2$, and $V(P_{0}) \cap V(H_{0}) \neq \emptyset$. 
We first show (\ref{crossing}) and (\ref{crossing2}). 
Suppose that 
the following {\rm(a)} or {\rm(b)} holds 
for some $v_{1} \in \iR{C}{u_{1}}{x_{1}}$ and some $v_{2} \in \iR{C}{u_{2}}{x_{2}}$: 
(a) There exists a $C$-path $P_{1}$ joining $v_{1}$ and $v_{2}$. 
(b) There exist disjoint $C$-paths $P_{2}$ joining $v_{l}$ and $w$, and $P_{3}$ joining $v_{3-l}$ and $w^{-}$ 
for some $l \in \{1, 2\}$ and some $w \in \iR{C}{v_{l}}{u_{3-l}}$. 
We choose such vertices $v_{1}$ and $v_{2}$ so that $|\IR{C}{u_{1}}{v_{1}}| + |\IR{C}{u_{2}}{v_{2}}|$ is as small as possible. 
Without loss of generality, 
we may assume that $l = 1$ if {\rm(b)} holds. 
Since $N_{C}(H_{0})\cap \{v_{1},v_{2}\}=\emptyset$,
$(V(P_{1}) \cup V(P_{2}) \cup V(P_{3})) \cap V(P_{0})= \emptyset$.
Therefore, we can define a cycle 
\begin{align*}
D = 
\left \{
\begin{array}{ll} 
P_{1}[v_{1},v_{2}] \IR{C}{v_{2}}{u_{1}} P_{0}[u_{1},u_{2}]\IR{\ola{C}}{u_{2}}{v_{1}} & \textup{if \rm{(a)} holds,} \\[1mm]
P_{2}[v_{1},w] \IR{C}{w}{u_{2}} \IL{P_{0}}{u_{2}}{u_{1}}\IR{\ola{C}}{u_{1}}{v_{2}}P_{3}[v_{2},w^{-}] \IR{\ola{C}}{w^{-}}{v_{1}} 
&\textup{otherwise.}
\end{array}
\right.
\end{align*}

For $i \in \{1, 2\}$,
let $Q_{i} = \ir{C}{u_{i}}{v_{i}}$. 
By Lemma \ref{non-insertible},
we can obtain the following statement (1), and
by the choice of $v_{1}$ and $v_{2}$, 
we can obtain the following statements (2)--(5): 
 
\noindent
\begin{enumerate}[{\upshape(1)}]

\item 
$N_G(x) \cap \ir{P_{0}}{u_{1}}{u_{2}} = \emptyset$ 
for $x \in V(Q_{1} \cup Q_{2})$. 

\item  
$N_G(x) \cap (\ir{P_{1}}{v_{1}}{v_{2}} \cup \ir{P_{2}}{v_{1}}{w} \cup \ir{P_{3}}{v_{2}}{w^{-}}) = \emptyset$ 
for $x \in V(Q_{1} \cup Q_{2})$.

\item  
$xy \notin E(G)$ for $x \in V(Q_{1})$ and $y \in V(Q_{2})$.

\item 
$I(x;C) \cap I(y;C) = \emptyset$ for $x \in V(Q_{1})$ and $y \in V(Q_{2})$. 

\item  
If (b) holds, then
$w^{-}w \not \in I(x;C)$ for $x \in V(Q_1 \cup Q_2)$.
\end{enumerate}

Let $a \in V(Q_i)$ for some $i \in \{ 1,2 \}$. 
Note that
each vertex of $Q_{i}$ is insertible,
that is,
$a \in X(C[u_{i}',u_{i}]) \cup Y(C(a,u_i])$.
We show that
$a \in X(D) \cup Y(Q_i(a,v_i) \cup D)$.
If $a \in X(C[u_{i}',u_{i}])$,
then
the statements (3) and (5)
yield that $a \in X(D)$.
Suppose that
$a \in Y(C(a,u_i])$.
By (3),
$N_G(a) \cap C(a,u_i] \subseteq  N_G(a) \cap \big( Q_i(a,v_i) \cup D \big)$.
This implies that 
$a \in Y(\ir{Q_{i}}{a}{v_i} \cup D)$.
By (1), (2) and (4),
$I(x;D) \cap I(y;D) = \emptyset$ for $x \in V(Q_{1})$ and $y \in V(Q_{2})$.
Thus, by Lemma \ref{D cup Q is hamilton},
$G[V(D \cup Q_1 \cup Q_2)]$ is Hamiltonian, 
which contradicts the maximality of $C$.

By using similar argument as above, we can also show (\ref{W1}) and (\ref{W2}). 
We only prove (\ref{W1}).
Suppose that for some $v_{1} \in \iR{C}{u_{1}}{x_{1}}$ and $v_{2} \in \iR{C}{u_{2}}{x_{2}}$, 
there exist disjoint $C$-paths 
$\IR{P_{1}}{v_{1}}{w_{1}}$, $\IR{P_{2}}{v_{2}}{w_{2}}$ and $\IR{P_{3}}{w_{1}^{-}}{w_{2}^{+}}$ 
with $w_{1} \in \ir{C}{v_{1}}{u_{2}}$ and $w_{2} \in \Ir{C}{w_{1}}{u_{2}}$. 
We choose such $v_{1}$ and $v_{2}$ so that $|\IR{C}{u_{1}}{v_{1}}| + |\IR{C}{u_{2}}{v_{2}}|$
is as small as possible.
Let $Q_{i} = \ir{C}{u_{i}}{v_{i}}$ for $i \in \{1, 2\}$. 
Then by Lemma \ref{insertible} (\ref{crossing}), 
$xy \notin E(G)$ 
for $x \in V(Q_{1})$ and $y \in V(Q_{2})$.   
By the choice of $v_{1}$ and $v_{2}$ and Lemma \ref{insertible} (\ref{crossing2}), 
$w_{1}w_{1}^{-}, w_{2}w_{2}^{+} \notin I(x;\IR{C}{v_1}{u_1}) \cup I(y;\IR{C}{v_2}{u_2})$
for $x \in V(Q_{1})$ and $y \in V(Q_{2})$.
By Lemma \ref{insertible} (\ref{crossing}) and (\ref{crossing2}), 
$I(x;\IR{C}{v_1}{u_2} \cup \IR{C}{v_2}{u_1}) \cap I(y;\IR{C}{v_1}{u_2} \cup \IR{C}{v_2}{u_1}) = \emptyset$
for $x \in V(Q_{1})$ and $y \in V(Q_{2})$.
Hence by applying Lemma \ref{D cup Q is hamilton} 
as 
$$D = \IR{P_{1}}{v_{1}}{w_{1}} \IR{C}{w_{1}}{w_{2}} \IR{\ola{P_{2}}}{w_{2}}{v_{2}}\IR{C}{v_{2}}{u_{1}}\IR{P_0}{u_{1}}{u_{2}}
\IR{\ola{C}}{u_{2}}{w_{2}^{+}}\IR{\ola{P_{3}}}{w_{2}^{+}}{w_{1}^{-}}\IR{\ola{C}}{w_{1}^{-}}{v_{1}},$$ 
$Q_{1}$ and $Q_{2}$,
we see that there exits a longer cycle than $C$, a contradiction.
\end{proof}

\section{Proof of Theorem \ref{main}}

\begin{proof}[Proof of Theorem \ref{main}]

The cases $k=1$, $k=2$ and $k=3$ were shown 
by Fraisse and Jung \cite{FJ},
by Bauer et al.~\cite{BBVL}
and by Ozeki and Yamashita \cite{OY}, respectively.
Therefore, we may assume that $k \geq 4$.
Let $G$ be a graph
satisfying the assumption of Theorem \ref{main}.
By Theorem \ref{degresult} (iii),
we may assume
$\alpha(G) \geq \kappa(G)+1$.
Let $C$ be a longest cycle in $G$.
If $C$ is a Hamiltonian cycle of $G$,
then there is nothing to prove.
Hence we may assume that $G \ms V(C)\not=\emptyset$.
Let $H=G \ms V(C)$ and $x_{0} \in V(H)$.
Choose a longest cycle $C$ and $x_{0}$
so that 
\begin{center}
$\dg{C}{x_0}$ is as large as possible.
\end{center}
Let $H_{0}$ be the component of $H$ such that $x_0 \in V(H_{0})$.
Let $$N_{C}(H_{0}) =U= \{u_1,u_{2},\ldots,u_{m}\}.$$
Note that $m \ge \kappa(G) \ge k$.
Let
$$M_{0}=\{0,1,\ldots,m\}\text{ and } M_{1}=\{1,2,\ldots,m\}.$$
Let $u_i'$ be the vertex in $N_{C}(H_{0})$
such that $\ir{C}{u_i}{u_i'}  \cap N_{C}(H_{0})=\emptyset$.
By Lemma \ref{non-insertible}, 
there exists a non-insertible vertex in $\ir{C}{u_{i}}{u_{i}'}$.
Let $x_i \in \ir{C}{u_{i}}{u_{i}'}$
be the first non-insertible vertex along the orientation of $C$ 
for each $i \in M_{1}$,
and let 
$$X=\{x_1,x_2,\ldots,x_m\}.$$
Note that $\dg{C}{x_{0}} \le |U|=|X|$.
Let $$D_i=\ir{C}{u_i}{x_i}
\text{ for each $i \in M_{1}$, and }
D=\bigcup_{i \in M_{1}}D_{i}.$$

\paragraph{}
We check the degree of $x_{i}$ in $C$ and $H$.
Since $x_i$ is non-insertible,
we can see that
\begin{equation} \label{C}
\dg{C}{x_{i}} \le |D_{i}| +\alpha(G)-1 \text{\ \ for $i \in M_{1}$.}
\end{equation}
By the definition of $x_{i}$, 
we clearly have 
$N_{H_{0}}(x_{i})=\emptyset$ for $i \in M_{1}$.
Moreover, by Lemma \ref{insertible} (i),
$N_{H}(x_i) \cap N_{H}(x_j) =\emptyset$
for $i,j \in M_{1}$ with $i\not=j$.
Thus we obtain
\begin{equation}
\sum_{i\in M_{0}}
\dg{H}{x_i} \le |H|-1, \label{H}
\end{equation}
and 
\begin{equation}
\sum_{i\in M_{1}}
\dg{H}{x_i} \le |H|-|H_{0}|. \label{H2}
\end{equation}

We check the degree sum in $C$ of two vertices in $X$.
Let $i$ and $j$ be distinct two integers in $M_{1}$.
In this paragraph,
we 
let $C_i = \IR{C}{x_i}{u_{j}}$
and $C_j = \IR{C}{x_j}{u_{i}}$.
By Lemma \ref{insertible} (ii),
we have
$N_{C_i}(x_i)^- \cap N_{C_i}(x_{j})=\emptyset$
and
$N_{C_j}(x_j)^- \cap N_{C_j}(x_{i})=\emptyset$.
By Lemma \ref{insertible} (i),
$N_{C_i}(x_i)^- \cup N_{C_i}(x_{j}) \subseteq C_{i} \sm D$,
$N_{C_j}(x_j)^- \cup N_{C_j}(x_{i}) \subseteq C_{j} \sm D$ 
and 
$N_{D_i}(x_j) = N_{D_j}(x_i)=\emptyset$.
Thus, we obtain
\begin{equation}\label{2vertices}
\dg{C}{x_i} +\dg{C}{x_j}
\le |C|-\sum_{h \in M_{1}\sm \{i,j\}}|D_{h}| \text{\ \ for  $i, j \in M_{1}$  with  $i \neq j$}.
\end{equation}
\bigskip

By Lemma \ref{insertible} (i) and since $N_{H_{0}}(x_{i}) = \emptyset$ for $i \in M_{1}$, 
we obtain the following.

\begin{Claim}\label{Xindep}
$X\cup \{x_{0}\}$ is an independent set,
and hence $|X| \le \alpha(G)-1$.
\end{Claim}

\begin{Claim}\label{k+1}
$|X| \ge \kappa(G)+1$.
\end{Claim}

\begin{proof}
Let $s$ and $t$ be distinct two integers in $M_{1}$.
By the inequality (\ref{2vertices}),
we have
$$\dg{C}{x_s} +\dg{C}{x_t}  \le |C|-\sum_{i \in M_{1}\sm \{s,t\}}|D_{i}|.$$
Let $I$ be a subset of $M_{0}$
such that
$|I|=k+1$ and $\{0,s,t\} \subseteq I$.
By Claim \ref{Xindep},
$\{x_{i} : i \in I\}$
is an independent set.
By the inequality (\ref{C}),
we deduce
\begin{eqnarray*}
\sum_{i \in I\sm \{0,s,t\}}\dg{C}{x_i}
&\le&
\sum_{i \in I\sm \{0,s,t\}}|D_{i}|+(k-2)(\alpha(G)-1).
\end{eqnarray*}
By the inequality (\ref{H}) and the definition of $I$, 
we obtain
\begin{eqnarray*}
\sum_{i \in I}\dg{H}{x_i}
&\le&
|H| - 1.
\end{eqnarray*}
Thus,
it follows from
these three inequalities
that
$$\sum_{i \in I}\dg{G}{x_i}
\leq n+(k-2)(\alpha(G)-1)-1+\dg{C}{x_0}.$$
Since $\sigma_{k+1}(G) \geq n+\kappa(G)+(k-2)(\alpha(G)-1)$,
we have
$|X| \ge \dg{C}{x_0} \ge \kappa(G)+1$.
\end{proof}

\paragraph{}
Let $S$ be a cut set with $|S|=\kappa(G)$, and
let $V_1,V_2, \ldots , V_p$
be the components of $G \ms S$.
By Claim \ref{k+1},
we may assume that
\begin{center}
there exists an integer $l$ such that
$C[u_{l}, u_{l}')   \subseteq V_{1}.$
\end{center}
By Lemma \ref{insertible} (i),
we obtain
\begin{equation}\label{xl}
\dg{C}{x_l}
\le |C \cap (V_1 \cup  S)|-|(\bigcup_{i \in M_{1} \sm \{l\}}D_{i} \cup X) \cap (V_1 \cup  S)|.
\end{equation}

\paragraph{}
By replacing the labels $x_2$ and $x_3$ if necessary,
we may assume that
$x_1$, $x_2$ and $x_3$ appear
in this order along the orientation of $C$.
In this paragraph,
the indices are taken modulo $3$.
From now we let $$C_i=\IR{C}{x_i}{u_{i+1}}$$
and
$$W_i :=\{w \in V(C_i) : \text{$w^+ \in N_{C_i}(x_i)$ and $w^- \in N_{C_i}(x_{i+1})$} \}$$
for each $i\in \{1,2,3\}$,
and let $W:=W_1 \cup W_2 \cup W_3$
(see Figure \ref{DefW}).
Note that $W \cap (U \cup \{x_{1}, x_{2}, x_{3}\}) =\emptyset$,
by the definition of $C_{i}$ and $W_{i}$ and by Lemma \ref{insertible} (i).

\begin{figure}[h]
\begin{center}
\includegraphics{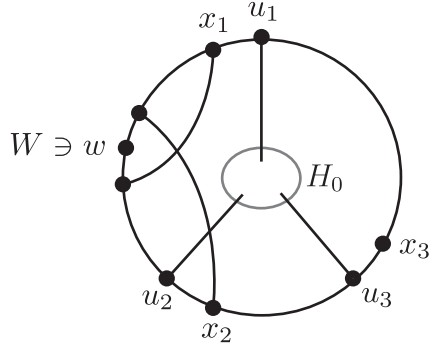}
\caption{The definition of $W$.}
\label{DefW}
\end{center}
\end{figure}

\begin{Claim}\label{U1}
$D \cup X \cup W \cup H \subseteq V_{1} \cup S$.
In particular,
$x_0 \in V_1 \cup S$.
\end{Claim}

\begin{proof}
We first show that $D \cup X \cup W \subseteq V_{1} \cup S$.
Suppose not.
Without loss of generality,
we may assume that
there exists an integer $h$ in $M_{1} \sm \{l\}$ such that
$\big(D_{h} \cup \{x_{h}\} \cup (W \cap \ir{C}{x_{h}}{u_{h}'})\big) \cap V_2\not=\emptyset,$ 
say $v \in 
\big(D_{h} \cup \{x_{h}\} \cup (W \cap \ir{C}{x_{h}}{u_{h}'})\big) \cap V_2$. 
Since $v \in V_{2}$,
it follows from Lemma \ref{insertible} (i) and (ii) 
that
$$
\dg{C}{v}
\le |C \cap (V_2 \cup  S)|-|(\bigcup_{i \in M_{1} \sm \{h\}} D_{i} \cup X) \cap (V_2 \cup  S)|.
$$

Let $I$ be a subset of $M_{0}\sm\{h\}$
such that
$|I|=k$ and $\{0,l\} \subseteq I$.
By Claim \ref{Xindep} and Lemma \ref{insertible} (i) and (ii),
$\{x_i : i \in I\} \cup \{v\}$ is an independent set
of order $k+1$.
By the above inequality
and the inequality (\ref{xl}),
we obtain 
\begin{eqnarray*}
\lefteqn{\dg{C}{x_l} +\dg{C}{v}}\\
& \le &
|C \cap (V_1 \cup V_2 \cup  S)|
+|C \cap S|
-|(\bigcup_{i \in M_{1} \sm \{l,h\}} D_{i} \cup X)\cap (V_1 \cup V_{2} \cup S)|\\
&=& |C| + |C \cap S| - |C \cap (\bigcup_{3 \le j \le p}V_j)|
-|(\bigcup_{i \in M_{1} \sm \{l,h\}} D_{i} \cup X)\cap (V_1 \cup V_{2} \cup S)|\\
&\le& |C| + |C \cap S| - |(\bigcup_{i \in M_{1} \setminus \{l, h\}}D_{i} \cup X) \cap (\bigcup_{3 \le j \le p}V_{j})|\\
&&{}-|(\bigcup_{i \in M_{1} \sm \{l,h\}} D_{i} \cup X)\cap (V_1 \cup V_{2} \cup S)|\\
&\le& |C| + \kappa(G) -\sum_{i \in M_{1} \sm \{l,h\}}|D_{i} \cap (\bigcup_{1 \le j \le p}V_{j} \cup S)|
-|X \cap (\bigcup_{1 \le j \le p}V_{j} \cup S)|\\
&\le& |C| + \kappa(G)  -\sum_{i\in I \sm \{0,l\}}|D_{i}|- |X|\\
&\le& |C| + \kappa(G)  -\sum_{i\in I \sm \{0,l\}}|D_{i}|- \dg{C}{x_0}.
\end{eqnarray*}
On the other hand,
the inequality (\ref{C}) yields that
$$\sum_{i\in I \sm \{0,l\}}\dg{C}{x_i} \le \sum_{i\in I \sm \{0,l\}}|D_{i}|+(k-2)(\alpha(G)-1).$$
By the above two inequalities, 
we deduce
$$\sum_{i \in I}\dg{C}{x_i}+\dg{C}{v}
\le |C| +\kappa(G)+(k-2)(\alpha(G)-1).$$
Recall that 
$\{x_i : i \in I\} \cup \{v\}$ is an independent set, 
in particular, $x_{0}\not\in \bigcup_{i \in I} N_{H}(x_{i}) \cup N_{H}(v)$. 
Since 
$ N_{H}(x_{i}) \cap N_{H}(x_{j}) = \emptyset$ for $i, j \in I$ with $i \neq j$
and
$(\bigcup_{i \in I} N_{H}(x_{i})) \cap N_{H}(v)=\emptyset$ 
by Lemma \ref{insertible} (i) and (ii), 
it follows that $\sum_{i \in I}\dg{H}{x_i}+\dg{H}{v} \le |H|-1$.
Combining this inequality with the above inequality, 
we get 
$\sum_{i \in I}\dg{G}{x_i}+\dg{G}{v}
\le n+\kappa(G)+(k-2)(\alpha(G)-1)-1$, 
a contradiction.

\paragraph{}
We next show that 
$H \ms H_{0} \subseteq V_1 \cup S$.
Suppose not.
Without loss of generality,
we may assume that 
there exists a vertex $y \in (H \ms H_{0}) \cap V_2$.
Let $H_{y}$ be a component of $H$ with $y \in V(H_{y})$.
Note that $H_{y}\not=H_{0}$.
Suppose that
$N_C(H_{y}) \cap (D_{h} \cup \{ x_h \}) \neq \emptyset$ for some $h \in M_{1} \sm \{l\}$.
Then
Lemma \ref{insertible} (i) yields that
$$d_C(y) \le |C \cap (V_2 \cup S)|-|(\bigcup_{i \in M_1 \sm \{ h \}} D_i \cup X) \cap (V_2 \cup S)|.$$
Hence, by the same argument as above,
we can obtain a contradiction. 
Thus
we may assume that 
$N_C(H_{y}) \cap (D_{i} \cup \{ x_i \}) = \emptyset$ for all $i \in M_{1} \sm \{l\}$.
Then,
since $y \in V_{2}$ and $D_{l} \cup \{x_{l}\} \subseteq V_{1}$,
we have
$$d_C(y) \le |C \cap (V_2 \cup S)|-|(\bigcup_{i \in M_1} D_i \cup X) \cap (V_2 \cup S)|.$$

Let $I$ be a subset of $M_{0}$
such that
$|I|=k$ and $\{0,l\} \subseteq I$.
Since
$x_{l} \in V_{1}$,
$y \in V_{2}$,
$H_{y}\not=H_{0}$
and
$N_C(H_{y}) \cap (D_{i} \cup \{ x_i \}) = \emptyset$ for all $i \in M_{1} \sm \{l\}$,
it follows from
Claim \ref{Xindep}
that $\{x_i : i \in I\} \cup \{y\}$ is an independent set of order $k+1$.
By the above inequality
and the inequality (\ref{xl}),
we obtain
\begin{eqnarray*}
\lefteqn{\dg{C}{x_l} +\dg{C}{y}}\\
& \le &
|C \cap (V_1 \cup V_2 \cup  S)|
+|C \cap S|
-|(\bigcup_{i \in M_{1} \sm \{l\}} D_{i} \cup X)\cap (V_1 \cup V_{2} \cup S)|\\
&\le& |C| +|C \cap S| -\sum_{i\in I \sm \{0,l\}}|D_{i}|- \dg{C}{x_0}.
\end{eqnarray*}
Therefore,
by the above inequality and the inequality (\ref{C}),
we obtain
$$\sum_{i \in I}\dg{C}{x_i}+\dg{C}{y}
\le |C|  +|C \cap S|+(k-2)(\alpha(G)-1).$$
Since 
$H_{0}\not=H_{y}$
and
$N_C(H_{y}) \cap (D_{i} \cup \{ x_i \}) = \emptyset$ for all $i \in M_{1}\sm \{l\}$,
it follows that
$(\bigcup_{i \in I \sm \{l\}} N_{H}(x_{i})) \cap V(H_{y})=\emptyset$.
Since
$x_{l} \in V_{1}$
and
$y \in V_{2}$,
we have
$N_{H}(x_{l}) \cap N_{H}(y) \subseteq H \cap S$.
Therefore,
we obtain 
$$\sum_{i \in I}\dg{H}{x_i}+\dg{H}{y} \le |H|+|H \cap S|-2.$$
Combining the above two inequalities, 
$\sum_{i \in I}\dg{G}{x_i}+\dg{G}{y}\le n+\kappa(G)+(k-2)(\alpha(G)-1)-2$, 
a contradiction.

\paragraph{}
We finally show that
$H_0 \subseteq V_1 \cup S$.
Suppose not.
Without loss of generality,
we may assume that
there exists a vertex $y_{0} \in H_{0} \cap V_{2}$.
Then
$$\dg{G}{y_{0}} \le |U \cap (V_{2} \cup S)|+ |H_{0}|-1.$$
Since $u_{l} \in V_{1}$,
we have $H_{0} \cap S \not=\emptyset$.
Note that by the above argument,
$X  \subseteq V_{1} \cup S$.
Therefore,
by Claim \ref{k+1},
$|X \cap V_{1}| = |X|-|X \cap S|
\ge \kappa(G)+1-(|S|-|H_{0}\cap S|)
\ge \kappa(G)+1-(\kappa(G)-1) = 2$.
Let $x_{s} \in X \cap V_{1}$ with $x_{s}\not=x_{l}$.
Let $I$ be a subset of $M_{1}$
such that
$|I|=k$ and $\{l,s\} \subseteq I$. 
Then
$\{x_i : i \in I\} \cup \{y_{0}\}$ is an independent set
of order $k+1$.
By Lemma \ref{insertible} (i),
we have
$N_{C}(x_{l})^- \cap (U \setminus \{u_{l}\}) =\emptyset$
and
$N_{C}(x_{s})^- \cap (U \setminus \{u_{s}\}) =\emptyset$.
Since  $x_{l}, x_{s} \in V_{1}$,
it follows that
$(N_{C}(x_{l}) \cup N_{C}(x_{s})) \cap (U \cap V_{2}) =\emptyset$.
Therefore,
we can improve the inequality (\ref{2vertices}) as follows:
\begin{equation*}
\dg{C}{x_l} +\dg{C}{x_s} \le |C|-\sum_{i \in I\sm \{l,s\}}|D_{i}|-|U \cap V_{2}|.
\end{equation*}
By the inequality (\ref{C}) and the inequality (\ref{H2}),
$$\sum_{i \in I \sm \{l,s\}}\dg{C}{x_i}
\le  \sum_{i \in I\sm \{l,s\}}|D_{i}|+(k-2)(\alpha(G)-1)
\text{\ \ and\ \ }
\sum_{i \in I }\dg{H}{x_i}
\le  |H|-|H_{0}|.$$
Hence, by the above four inequalities,
we deduce
$\dg{G}{y_{0}}+\sum_{i \in I}\dg{G}{x_i}
\le n+\kappa(G)+(k-2)(\alpha(G)-1)-1,$
a contradiction.
\end{proof}

\paragraph{}
By Claim \ref{U1},
\begin{center}
there exists an integer $r$
such that
$\iR{C}{x_{r}}{u_{r}'} \cap \bigcup_{i=2}^{p}V_{i}\not=\emptyset$,
\end{center}
say 
$$v_{2} \in \iR{C}{x_{r}}{u_{r}'} \cap \bigcup_{i=2}^{p}V_{i}.$$
Choose $r$ and $v_{2}$ so that 
$v_{2}\not= u_{r}'$
if possible.
Without loss of generality,
we may assume that
$v_{2} \in V_{2}.$
Note that
\begin{equation}
\dg{G}{v_{2}} \le |V_{2} \cup S|-1.\label{v2}
\end{equation}

\begin{Claim}\label{dgw}
$\dg{C}{w}\le \dg{C}{x_{0}} \le |X| \le \alpha(G)-1$
for each $w \in  W$.
\end{Claim}

\begin{proof}
Let $w \in  W$.
Without loss of generality,
we may assume that
$w \in W_1$. 
Then by applying Lemma \ref{D cup Q is hamilton} 
as $Q_{1} = D_{1}$, $Q_{2} = D_{2}$ and 
$$D = x_1\IR{C}{w^{+}}{u_{2}}\IR{P}{u_2}{u_1}\IL{C}{u_1}{x_{2}}\IL{C}{w^{-}}{x_1},$$ 
where $\IR{P}{u_2}{u_1}$ is a $C$-path 
passing through some vertex of $H_{0}$, 
we can obtain a cycle $C'$ such that $V(C) \setminus \{w\} \subseteq V(C')$ and $V(C') \cap V(H_{0}) \neq \emptyset$ 
(note that (I) and (II) of Lemma \ref{D cup Q is hamilton} hold,
by Lemma \ref{insertible} (i) and (ii) and the definition of insertible and $D_{i}$). 
Note that by the maximality of $|C|$, $|C'| = |C|$. 
Note also that $\dg{C'}{w} \ge \dg{C}{w}$. 
By the choice of $C$ and $x_{0}$, 
we have 
$\dg{C'}{w} \le \dg{C}{x_{0}}$, 
and hence 
by Claim \ref{Xindep} and the fact that $\dg{C}{x_{0}} \le |X|$, 
we obtain 
$\dg{C}{w}\le \dg{C}{x_{0}} \le |X| \le \alpha(G)-1$.
\end{proof}

By Lemma \ref{insertible} 
and Claim \ref{U1},
we have
\begin{equation}\label{xwH}
\sum_{i \in M_{0}}\dg{H}{x_{i}}+\sum_{w \in W}\dg{H}{w} \le |H|-|\{x_{0}\}| =|H \cap (V_1 \cup S)|-1.
\end{equation}

Moreover,
by Lemma \ref{insertible} and Claim \ref{Xindep},
the following claim holds. 

\begin{Claim}\label{XW}
$X \cup W \cup \{x_{0}\}$ is an independent set.
\end{Claim}

We now check the degree sum of the vertices $x_{1},x_{2}$ and $x_{3}$ in $C$.
In this paragraph,
the indices are taken modulo $3$.
By Lemma \ref{insertible} (ii),
$(N_{C_i}(x_i)^- \cup N_{C_i}(x_{i+1})^+) \cap N_{C_i}(x_{i+2})=\emptyset$
for $i\in \{1,2,3\}$.
Clearly,
$N_{C_i}(x_{i})^- \cap N_{C_i}(x_{i+1})^+=W_i$
and
${N_{C_i}(x_i)}^- \cup {N_{C_i}(x_{i+1})}^+
\cup {N_{C_i}(x_{i+2})} \subseteq C_i \cup
\{u_{i+1}^+\}$.
By Lemma \ref{insertible} (i),
$(N_{C_{i}}(x_i)^{-} \cup N_{C_{i}}(x_{i+2})) \cap D_j=\emptyset$
for $i \in \{1,2,3\}$ and $j\in M_{1}$.
For $i \in \{1,2,3\}$,
let
$$L_{i}=\left\{ x_{j} \in X \sm \{x_{i+1}\} : N_{C_{i}}(x_{i+1})^{+} \cap D_{j}\not=\emptyset \right\}$$
and
let
$L=\bigcup_{i \in \{1,2,3\}}L_{i}$
(see Figure \ref{DefL}).

\begin{figure}[h]
\begin{center}
\includegraphics{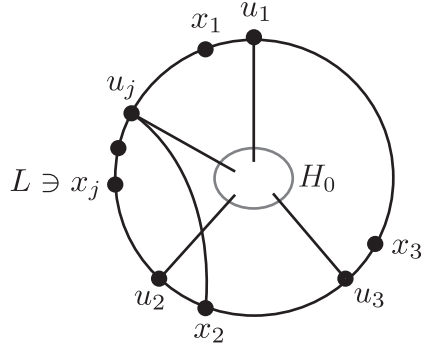}
\caption{The definition of $L$.}
\label{DefL}
\end{center}
\end{figure}
Note that $L \cap \{x_1,x_2,x_3\}=\emptyset$
and $W \cap L=\emptyset$
by Lemma \ref{insertible} (i).
Therefore
 the following inequality holds:
$$\dg{C_i}{x_1}+ \dg{C_i}{x_2}+\dg{C_i}{x_3}
\leq |C_i| + |W_i|+ 1-\sum_{j \in M_{1}}|C_i \cap D_{j}|+ |L_{i}|$$
for $i \in\{1,2,3\}$.
By Lemma \ref{insertible} (i),
we have
$N_{C}(x_{i}) \cap D_{j}=\emptyset$
for $i,j \in M_{1}$ with $i\not=j$,
and hence
$$\dg{D_i}{x_1}+ \dg{D_i}{x_2}+\dg{D_i}{x_3}
\leq |D_i|$$
for $i \in\{1,2,3\}$.
Let $I$ be a subset of $M_{0}$
such that 
$I \cap \{1,2,3\}=\emptyset$.
Let $L_{I}=L \cap \{x_{i} : i \in I\}$.
Note that $|L \cap \{x_{i}\}| - |D_i| \le 0$ for each  $i \in M_{1} \sm \{1,2,3\}$.
Thus,
we deduce
\begin{eqnarray}
\dg{C}{x_1}+\dg{C}{x_2}+\dg{C}{x_3}
&\le&
\sum_{i=1}^3(|C_i|+|W_i|+|L_{i}|+1-\sum_{j \in M_{1}}|C_{i} \cap D_{j}|+|D_{i}|)\nonumber\\
&=&
|C|+|W|+|L|-\sum_{i \in M_{1}  \sm \{1,2,3\}}|D_{i}|+3\label{WL}\nonumber\\
&\le&
|C|+|W|+|L_{I}|-\sum_{i \in I\sm\{0\}}|D_{i}|+3\label{W3}\\
&\le&
|C|+|W|+3.\label{W}
\end{eqnarray}

\begin{Claim}\label{WLk-2}
$|W|+|L| \ge \kappa(G)-2 \ge 1$.
\end{Claim}

\begin{proof}
Let $I$ be a subset of $M_{0}$
such that $|I|=k-2$ and 
$I \cap \{1,2,3\}=\emptyset$.
Suppose that
$|W|+|L_{I}|\le \kappa(G)-3$.
By Claim \ref{XW}, 
$\{x_{i} : i \in I\}\cup \{x_{1},x_{2},x_{3}\}$ is an independent set of order $k+1$.
By the inequality (\ref{WL}),
we obtain
\begin{eqnarray*}
\dg{C}{x_1} +\dg{C}{x_2}+\dg{C}{x_3}
&\leq&
|C|+\kappa(G)-\sum_{i \in I\sm\{0\}}|D_{i}|.
\end{eqnarray*}
Therefore,
this inequality,
the inequalities (\ref{C}) and (\ref{H})
and Claim \ref{dgw}
yield that
\begin{eqnarray*}
\sum_{ i= 1}^{3}\dg{G}{x_i}+\sum_{ i\in I }\dg{G}{x_i}
&\le&
n+\kappa(G)+(k-2)(\alpha(G)-1)-1,
\end{eqnarray*}
a contradiction.
Therefore
$|W|+|L| \ge |W|+|L_{I}| \ge \kappa(G)-2$.
\end{proof}

\begin{Claim}\label{x0a-1}
$\dg{C}{x_{0}}=|U|=|X|=  \alpha(G)-1$.
In particular, $N_{C}(x_{0})=U$.
\end{Claim}

\begin{proof}
Suppose that $\dg{C}{x_{0}} \le \alpha(G)-2$. 
In this proof,
we assume $x_l =x_1$
(recall that $l$ is an integer such that $C[u_{l}, u_{l}')  \subseteq V_{1}$, 
see the paragraph below the proof of Claim \ref{k+1}).
We divide the proof into two cases. 

\bigskip

\noindent\textit{Case 1.} $|W| \ge k-3$.

\begin{subclaim}\label{W2k-5}
$|W| \le \kappa(G)+k-5$.
\end{subclaim}

\begin{proof}
Suppose that $|W| \ge \kappa(G)+k-4$.
By Claim \ref{U1},
we obtain
\begin{eqnarray*}
|(W \cup \{x_{0},x_{1},x_{2},x_{3}\}) \cap V_{1}|
&=&|W \cup \{x_{0},x_{1},x_{2},x_{3}\}|-|(W\cup \{x_{0},x_{1},x_{2},x_{3}\}) \cap S|\\
& \ge& (\kappa(G)+k-4+4) -\kappa(G) =k.
\end{eqnarray*}
Let $W'$ be a  subset of $(W\cup \{x_{0},x_{1},x_{2},x_{3}\}) \cap V_{1}$
such that $|W'|=k$ and  $x_{1} \in W'$.
Since $W' \subseteq V_{1}$
and $v_{2} \in V_{2}$,
it follows from Claim \ref{XW}
that
$W' \cup \{v_{2}\}$
is an independent set of order $k+1$.
By the inequality (\ref{xl}) and Claims \ref{U1} and \ref{dgw},
we obtain
\begin{eqnarray*}
\dg{C}{x_1}
&\le& 
|C \cap (V_1 \cup  S)|-\sum_{i \in M_{1} \sm \{1\}}|(D_{i} \cap (V_1 \cup  S) | - |X \cap (V_1 \cup  S)|\\
&\le& |C \cap (V_1 \cup  S)|-\sum_{i \in \{2,3\}}|D_{i}|-|X|\\
&\le& |C \cap (V_1 \cup  S)|-\sum_{i \in \{2,3\}}|D_{i}|-\dg{C}{w_0},
\end{eqnarray*}
where $w_{0} \in W'\sm\{x_{1},x_{2},x_{3}\}$
(note that $|W'| = k \ge 4$).
By the inequality (\ref{C}) and Claim \ref{dgw},
$$
\sum_{x \in W' \cap \{x_{2},x_{3}\}}\dg{C}{x}
+
\sum_{w \in W'\sm \{w_{0},x_{1},x_{2},x_{3}\}}\dg{C}{w}
\le \sum_{i \in \{2,3\}}|D_{i}|+(k-2)(\alpha(G)-1).$$
By the above two inequalities, 
we obtain
\begin{equation*}
\sum_{w \in W'}\dg{C}{w}
\le |C \cap (V_1 \cup  S)|
+(k-2)(\alpha(G)-1).
\end{equation*}
Therefore, 
since 
$\sum_{w \in W'}\dg{H}{w} \le |H \cap (V_1 \cup S)|-1$ 
by the inequality (\ref{xwH}), 
it follows that 
\begin{equation*}\label{notr}
\sum_{w \in W'}\dg{G}{w}
\le |V_1 \cup  S|+(k-2)(\alpha(G)-1)-1.
\end{equation*}
Summing this inequality and the inequality (\ref{v2})
yields that
$\sum_{w \in W'}\dg{G}{w}
+\dg{G}{v_{2}}
\le n+\kappa(G)+(k-2)(\alpha(G)-1)-2,$
a contradiction.
\end{proof}

By the assumption of Case 1,
we can take a subset $W^{*}$ of $W \cup \{x_{0}\}$
such that $|W^{*}|=k-2$.
By Claim \ref{XW},
$W^{*} \cup \{x_{1},x_{2},x_{3}\}$ is independent.
Moreover,
by Claim \ref{dgw} and 
the assumption that $\dg{C}{x_{0}}\leq \alpha(G) -2$, 
we have
$$\sum_{w \in W^{*}}\dg{C}{w} \le(k-2)(\alpha(G)-2).$$
By Subclaim \ref{W2k-5},
summing
this inequality
and
the inequality (\ref{W})
yields
that
\begin{eqnarray*}
&&\sum_{i=1}^{3}\dg{C}{x_i}
+\sum_{w \in W^{*}}\dg{C}{w}\\
&\le&
|C|+|W|+3+(k-2)(\alpha(G)-2)\\
&\le&
|C|+(\kappa(G)+k-5)+3-(k-2)+(k-2)(\alpha(G)-1)\\
&=&
|C|+\kappa(G)+(k-2)(\alpha(G)-1).
\end{eqnarray*}
Therefore, 
since 
$\sum_{i=1}^{3}\dg{H}{x_i}+\sum_{w \in W^{*}}\dg{H}{w} \le |H|-1$ 
by the inequality (\ref{xwH}),
we obtain
$\sum_{i=1}^{3}\dg{G}{x_i}+\sum_{w \in W^{*}}\dg{G}{w}
\le n+\kappa(G)+(k-2)(\alpha(G)-1)-1$,
a contradiction.
\bigskip

\noindent\textit{Case 2.} $|W| \le k-4$.

By Claim \ref{WLk-2},
we can take a subset
$L^{*}$ of $L$
such that
$|L^{*}|=k-3-|W|$.
Let $I=\{ i : x_{i} \in L^{*}\} $.
By Claim \ref{XW},
$W \cup L^{*}\cup \{x_{0}, x_{1},x_{2},x_{3}\}$ is 
an independent set of order $k+1$.
By the inequality (\ref{WL}),
we have
\begin{eqnarray*}
\dg{C}{x_1}+\dg{C}{x_2}+\dg{C}{x_3}
&\le&
|C|+|W|+|L^{*}|-\sum_{i \in I}|D_{i}|+3\\
&=&
|C|+k-3-\sum_{i \in I}|D_{i}|+3\\
&\le&
|C|+\kappa(G)-\sum_{i \in I}|D_{i}|.
\end{eqnarray*}

On the other hand,
it follows from Claim \ref{dgw}, the assumption $\dg{C}{x_{0}}\leq \alpha -2$ and the inequality (\ref{C}) that
\begin{eqnarray*}
\sum_{w \in W \cup \{x_{0}\}}\dg{C}{w}+
\sum_{x \in L^{*}}\dg{C}{x}
&\le&
(|W|+1)(\alpha(G)-2)
+\sum_{i \in I }|D_{i}|+|L^{*}|(\alpha(G)-1)\\
&=&
(k-2)(\alpha(G)-1)
-|W|-1
+\sum_{i \in I }|D_{i}|\\
&\le&
(k-2)(\alpha(G)-1)
+\sum_{i \in I }|D_{i}|-1.
\end{eqnarray*}
Thus,
we deduce
$$\sum_{i=1}^{3}\dg{C}{x_i}+
\sum_{w \in W \cup \{x_{0}\}}\dg{C}{w}+
\sum_{x \in L^{*}}\dg{C}{x}
\le
|C|+\kappa(G)+(k-2)(\alpha(G)-1)-1.$$
By the inequality (\ref{xwH}),
we obtain
$$\sum_{i=1}^{3}\dg{H}{x_i}
+\sum_{w \in W \cup \{x_{0}\}}\dg{H}{w}+
\sum_{x \in L^{*}}\dg{H}{x}
\le |H|-1.
$$
Summing the above two inequalities
yields that
$\sum_{i=1}^{3}\dg{G}{x_i}+
\sum_{w \in W \cup \{x_{0}\}}\dg{G}{w}+
\sum_{x \in L^{*}}\dg{G}{x}
\le n+\kappa(G)+(k-2)(\alpha(G)-1)-2,$
a contradiction.
\bigskip

By Cases 1 and 2,
we have
$\dg{C}{x_{0}} \ge \alpha(G)-1$.
Since
$|U|=|X|$,
it follows 
from Claim \ref{dgw}
that $\dg{C}{x_{0}} =|U|=|X| =  \alpha(G)-1$. 
In particular, $N_{C}(x_{0})=U$ 
because $N_{C}(x_{0}) \subseteq N_{C}(H_{0}) = U$. 
This completes the proof of Claim \ref{x0a-1}.
\end{proof}

\begin{Claim}\label{WX}
$W \subseteq X$. 
\end{Claim}
\begin{proof}
If $W \setminus X \neq \emptyset$, 
then by Claim \ref{XW}, 
we have $\dg{C}{x_{0}} \le |X| \le \alpha(G)-2$, 
which contradicts Claim \ref{x0a-1}.
\end{proof}

\begin{Claim} \label{c}
If there exist distinct two integers $s$ and $t$ in $M_1$
such that
$u_s \in N_C(x_t)$,
then 
$N_C(x_s) \cap \IR{C}{u_{t}}{u_{s}} \subseteq U $.
\end{Claim}

\begin{proof}
Suppose that there exists
a vertex $z \in N_C(x_s) \cap C[u_{t},u_{s}]$ such that $z \not \in U$.
We show that
$X \cup \{x_0,z^+ \}$ is an independent set of order $|X|+2$.
By Claim \ref{XW},
we only show that $z^+ \not \in X$
and $z^+ \not\in N_{C}(x_{i})$
for each $x_i \in X \cup \{x_{0}\}$.
Since $z \not \in U$,
it follows from Lemma \ref{insertible} (i)
that $z^+ \not \in X$.
Suppose that $z^+ \in N_C(x_h)$ for some $x_h \in X \cup \{ x_0 \}$.
Since $x_s$ is a non-insertible vertex, 
it follows that $x_h \neq x_s$.
Let $z_{s}$ be the vertex in $C(u_{s},x_{s}]$
such that
$z \in N_{G}(z_{s})$
and
$z \not\in N_{G}(v)$ for all $v \in C(u_{s},z_{s})$.
By Lemma \ref{insertible} (ii),
we obtain $x_h \not \in C[u_s',z]$.
Therefore, 
$x_h \in C(z,u_s] \cup \{x_{0}\}$. 
If $x_h \in C(z,u_s]$,
then
we let $z_{h}$ be the vertex in $C(u_{h},x_{h}]$
such that 
$z^{+} \in N_{G}(z_{h})$
and
$z^{+} \not\in N_{G}(v)$ for all $v \in C(u_{h},z_{h})$.
We define the cycle $C^{*}$ as follows (see Figure \ref{Xxzh-fig}):
$$C^{*}
=
\begin{cases}
z_s \ola{C} [z,x_t] \ola{C} [u_s,z_h]C[z^+,u_h] x_0 \ola{C}[u_t,z_s]
&\text{if $x_h \in C(z,u_s]$,}\\
z_s\ola{C}[z,x_t]\ola{C}[u_s,z^+]x_h\ola{C}[u_t,z_s]
&\text{if $x_{h}=x_{0}$.}
\end{cases}
$$

\begin{figure}[h]
\begin{center}
\includegraphics{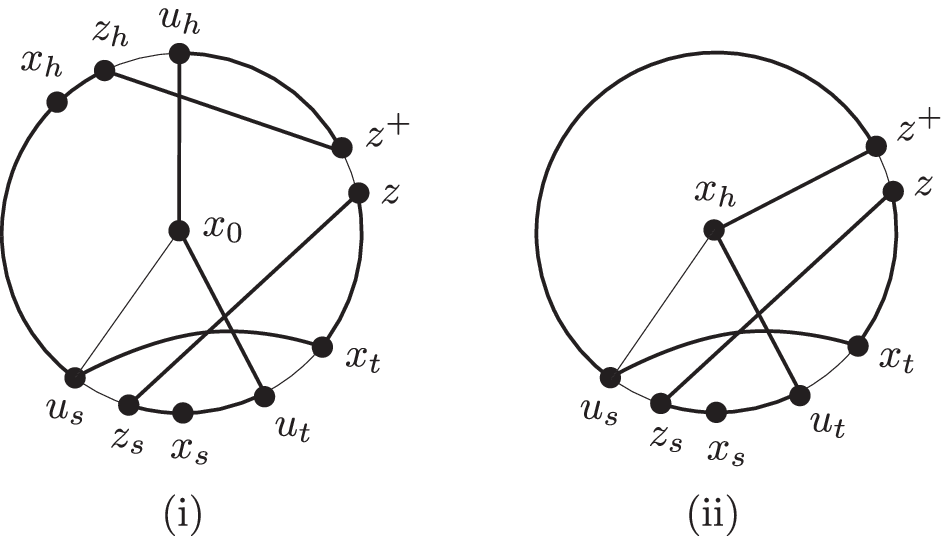}
\caption{Claim \ref{c}}
\label{Xxzh-fig}
\end{center}
\end{figure}

Then,
by similar argument in the proof of Lemma \ref{insertible},
we can obtain a longer cycle than $C$
by inserting all vertices of $V(C \sm C^{*})$ into $C^{*}$. 
This contradicts that $C$ is longest.
Hence $z^+ \not\in N_{C}(x_{h})$
for each $x_h \in X \cup \{x_{0}\}$.
Thus,
by Claim \ref{x0a-1},
$X \cup \{x_0,z^+ \}$ is an independent set of order $|X|+2 = \alpha(G)+1$,
a contradiction.
\end{proof}

\paragraph{}We divide the rest of the proof into two cases.

\bigskip

\noindent\textbf{Case 1. $v_{2} \not\in U$.}
\medskip

Let $Y=N_{G}(v_{2}) \cap X$,
and
let $\gamma=|X|-\kappa(G)-1$.
Note that $|X| = \kappa(G)+\gamma+1 \ge k+\gamma+1$
and
$x_{l}\not\in Y$ since $x_{l}\in V_{1}$.

\begin{Claim} \label{ge3}
$|Y| \ge \gamma+3$.
\end{Claim}

\begin{proof}
Suppose that
$|Y| \le \gamma+2$.
By the assumption of Case 1,
we have
$x_{0}v_{2} \not\in E(G)$.
Since
$|M_{0}| = |X|+1 \ge k+\gamma+2$ 
and $|Y| \le \gamma+2$,
there exists a subset $I$ of $M_{0} \sm \{i :x_{i} \in Y\}$
such that
$|I|=k$
and
$\{0,l\} \subseteq I$.
Then
$\{x_{i} : i \in I\} \cup \{v_{2}\}$
is an independent set of order $k+1$.
By the inequality (\ref{xl})
and Claims \ref{U1} and \ref{x0a-1},
we obtain
\begin{eqnarray*}
\dg{C}{x_l}
&\le& |C \cap (V_1 \cup  S)|-\sum_{i \in I \sm \{0,l\}}|D_{i}|-|X|\\
&=& |C \cap (V_1 \cup  S)|-\sum_{i \in I \sm \{0,l\}}|D_{i}|-\dg{C}{x_0}.
\end{eqnarray*}
Therefore 
it follows from the inequality  (\ref{C}) that
\begin{equation*}
\sum_{i \in I}\dg{C}{x_i}
\le |C \cap (V_1 \cup  S)|
+(k-2)(\alpha(G)-1).
\end{equation*}
By the inequality (\ref{xwH}),
$\sum_{i \in I}\dg{H}{x_i} \le |H \cap (V_1 \cup S)| -1$.
Summing these two inequalities
and the inequality (\ref{v2})
yields that
\begin{equation*}\label{notr}
\sum_{i \in I}\dg{G}{x_i} +\dg{G}{v_{2}}
\le n+ \kappa(G) +(k-2)(\alpha(G)-1)-2,
\end{equation*}
a contradiction.
\end{proof}

Recall that $r$ is an integer such that $v_{2} \in \iR{C}{x_{r}}{u_{r}'} \cap V_{2}$
(see the paragraph below the proof of Claim \ref{U1}). 
In the rest of Case 1, we assume that $l=1$. 
If $u_{r}' \not= u_{1}$, then let $r = 2$ and $u_{3} = u_{2}'$; 
otherwise, let $r = 3$ and let $u_{2}$ be the vertex with $u_{2}' = u_{3}$.

By Claim \ref{WX},
we have $W \subseteq X$.
Hence
we obtain $Y \cup W \cup L \subseteq X \sm \{x_{1}\}$. 
Recall that $W \cap L = \emptyset$. 
Therefore,
by Claims \ref{WLk-2} and \ref{ge3},
 we obtain
\begin{eqnarray*}
|Y \cap (W \cup L)|
&=& |Y|+|W|+|L|-|Y \cup (W \cup L)|\\
&\ge& \gamma+3+\kappa(G)-2-|X\sm\{x_{1}\}|\\
&=& \gamma+3+\kappa(G)-2-((\kappa(G)+\gamma+1)-1) = 1.
\end{eqnarray*}
Hence there exists a vertex $x_{h} \in Y \cap (W \cup L)$,
that is,
$v_{2}  \in N_{C}(x_{h}) \sm U$.
Since $\ir{C}{x_{2}}{x_{3}} \cap X =\emptyset$
and $\ir{C}{x_{3}}{x_{1}} \cap X =\emptyset$ if $r=3$,
either
$u_{h}\in N_{C}(x_{1})$ and $u_{h}\in C(x_{3},u_{1})$
or
$u_{h}\in N_{C}(x_{2})$ and $u_{h}\in C(x_{1},u_{2})$
holds
(especially, if $r=3$ then $u_h \in N_C(x_2)$ and $u_h \in C(x_1,u_2)$ holds)
 (see Figure \ref{case1}).

\begin{figure}[h]
\begin{center}
\includegraphics{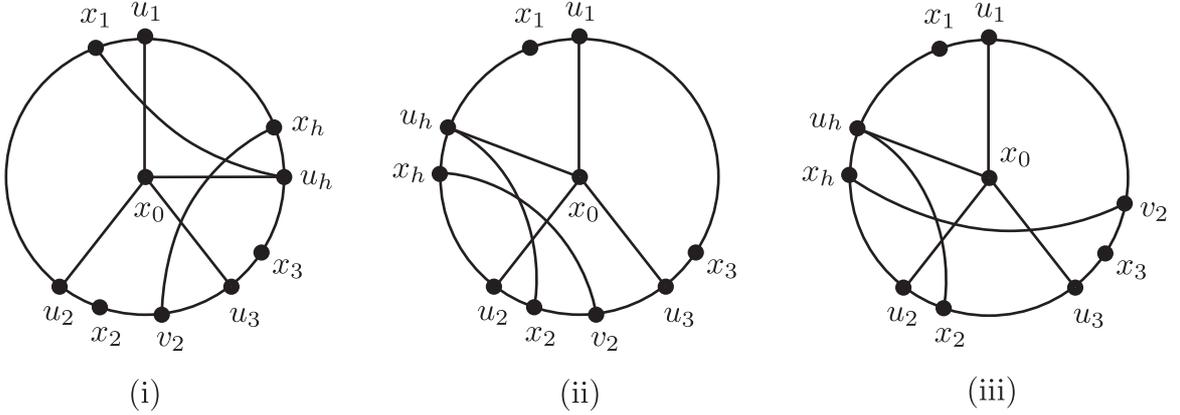}
\caption{The case ${r} = {2}$ and  the case ${r} = {3}$.}
\label{case1}
\end{center}
\end{figure}

If ${r} = {2}$ and $u_{h}  \in N_{C}(x_{1})$,
then
$v_{2} \in \IR{C}{u_{1}}{u_{h}}$
(see Figure \ref{case1} (i)).
If ${r} = {2}$ and $u_{h}  \in N_{C}(x_{2})$,
then
$v_{2} \in \IR{C}{u_{2}}{u_{h}}$
(see Figure \ref{case1} (ii)).
If ${r} = {3}$,
then
$u_{h} \in N_{C}(x_{2})$
and
$v_{2} \in \IR{C}{u_{2}}{u_{h}}$
(see Figure \ref{case1} (iii)).
In each case,
we obtain a contradiction to Claim \ref{c}.

\bigskip

\noindent\textbf{Case 2. $v_{2} \in U$.}
\medskip

We rename $x_i \in X$ for $i \ge 1$ as follows (see Figure \ref{jumping-fig}):
Rename an arbitrary vertex of $X$ as $x_{1}$.
For $i \ge 1$,
we rename
$x_{i+1} \in X$  
so that
$u_{i+1} \in N_C(x_{i}) \cap (U \sm \{ u_{i} \})$
and
$|C[u_{i+1},x_{i})|$ is as small as possible.
(For $x_{i} \in X$,
let $x_i'$ and $x_{i}''$ be the successors of $x_{i}$ and  $x_{i}'$ in $X$ along the orientation of $C$, respectively. 
Then by applying Claim \ref{WLk-2} as $x_1=x_{i}$, $x_2=x_{i}'$ and $x_3=x_i''$, 
it follows that $W \cup L \neq \emptyset$. 
By the definition of $x_{i}', x_{i}''$ and Claim \ref{WX}, 
we have $W_{1} = W_{2} = \emptyset$ (note that $W \cap \{x_{1}, x_{2}, x_{3}\} = \emptyset$). 
By the definitions of $x_{i}', x_{i}'', L_{1}$ and $L_{2}$, 
we also have $L_{1} = L_{2} = \emptyset$. 
Thus $W_{3} \cup L_{3} \neq \emptyset$. 
By Lemma \ref{insertible} (i) and since $W \cup L \subseteq X$, 
this implies that $N_C(x_{i}) \cap (U \sm \{ u_{i} \}) \neq \emptyset$.)
Let $h$  be the minimum integer
such that 
$x_{h+1} \in C(x_h,x_1]$. 
Note that this choice implies $h \geq 2$.
We rename $h$ vertices in $X$
as $\{x_1,x_{2},\ldots,x_{h}\}$ as above,
and
$m-h$ vertices in $X \sm \{x_{1}, x_{2},\ldots,x_{h}\}$
as $\{x_{h+1}, x_{h+2}, \ldots, x_{m}\}$ arbitrarily.
Let 
\begin{align*}
A_1=A_{h+1} = C[x_{1},x_{h})
\textup{ and }
A_i= C[x_{i},x_{i-1}) 
\textup{ for } 
2 \le i \le h. 
\end{align*}
Let
$$U_{1}=\{u_i \in U : x_i \in X \cap V_{1}\}.$$
If possible,
choose $x_1$ so that
$A_{2} \cap U_1 = \emptyset$. 

\begin{figure}[h]
\begin{center}
\scalebox{1}{\includegraphics{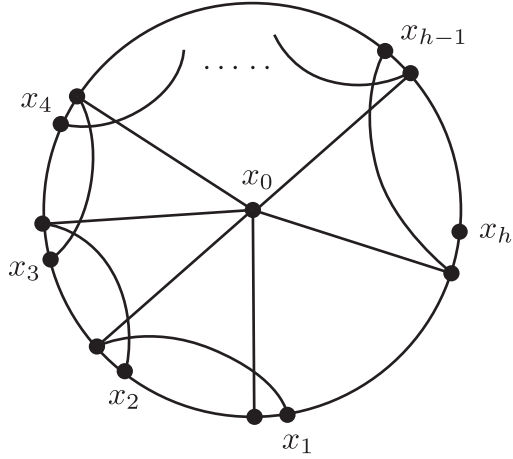}}
\caption{The choice of $\{x_1, \dots, x_{h}\}$.}
\label{jumping-fig}
\end{center}
\end{figure}

\paragraph{}
We divide the proof of Case 2 according to whether $h \le k$ or $h \ge k+1$.

\bigskip
\noindent\textbf{Case 2.1.} $h \le k$.
\medskip

By the choice of 
$\{x_1, \dots, x_{h}\}$,
we have
\begin{eqnarray}
\mbox{$N_{A_{i+1}}(x_i) \cap U \subseteq \{ u_i \}$ for $1 \le i \le h$.}\label{tonari}
\end{eqnarray}
By Claim \ref{c} and (\ref{tonari}),
we obtain
\begin{eqnarray}
\mbox{$N_{C \sm A_{i}}(x_i) \subseteq  (U \sm (A_i  \cup A_{i+1})) \cup D_i \cup \{ u_i \}$ for $2 \le i \le h$.}\label{sotogawa}
\end{eqnarray}

By Lemma \ref{insertible} (i) and (ii),
$N_{A_{i}}(x_i)^{-}  \cap N_{A_{i}}(x_{1})=\emptyset$
for $2 \le i \le h$.
By Lemma \ref{insertible} (i),
we have
$N_{A_{i}}(x_i)^{-}  \cup N_{A_{i}}(x_{1}) \subseteq A_i \sm D$
for $3 \le i \le h$.
Thus, 
it follows from (\ref{sotogawa})
that for $3 \le i \le h$
\begin{eqnarray*}
\mbox{$d_C(x_i) \le (|U|-|(A_{i} \cup A_{i+1}) \cap U| +|D_i|+1) + (|A_{i}|-|A_i \cap D|-\dg{A_{i}}{x_{1}})$.}
\end{eqnarray*}
By Lemma \ref{insertible} (i) and (\ref{tonari}),
we have
$N_{A_{2}}(x_2)^{-}  \cup N_{A_{2}}(x_{1}) \subseteq (A_2 \sm (U \cup D)) \cup D_1 \cup  \{ u_1 \}$.
Thus, 
by (\ref{sotogawa}),
we have
\begin{eqnarray*}
d_C(x_2) &\le& (|U|- |(A_{2} \cup A_{3}) \cap U|+|D_2|+1 )\\ 
&&{}+ (|A_2|-|A_2 \cap (U \cup D)|+|D_1|+1-\dg{A_{2}}{x_{1}}).
\end{eqnarray*}
Since $|A_1 \cap X|=|A_1 \cap U|$,
it follows from Lemma \ref{insertible} (i) that
\begin{eqnarray*}
d_{A_{1}}(x_1)
&\le&
|A_{1}|-|A_1 \cap D|-|A_1 \cap X|\\
&=&
|A_{1}|-|A_1 \cap D|-|A_1 \cap U|.
\end{eqnarray*}

By Claim \ref{x0a-1},
$d_{C}(x_0) =  |U| = \alpha(G)-1$.
Thus,
since $h \le k$,
we obtain
\begin{eqnarray*}
\sum_{0 \le i \le h}d_C(x_i)
&\le& 
\sum_{1 \le i \le h}|A_i|+
h|U|
- 2\sum_{1 \le i \le h}|A_i\cap U| +h
+\sum_{1 \le i \le h}|D_{i}| -\sum_{1 \le i \le h}|A_i \cap D|\\
&=& 
|C|+
(h-2)|U|
+h
+\sum_{1 \le i \le h}|D_{i}|-|D|\\
&\le&
|C|+k+(h-2)(\alpha(G)-1)
+\sum_{1 \le i \le h}|D_{i}|-|D|.
\end{eqnarray*}

Let $I$ be a subset of $M_{0}$
such that
$|I|=k+1$
and
$\{0,1,\ldots, h\} \subseteq I$.
By Claim \ref{XW},
$\{x_i : i \in I\}$ is an independent set
of order $k+1$.
By the above inequality and the inequality (\ref{C}),
we have \begin{eqnarray*}
\sum_{i \in I}d_C(x_i)
&\le& |C|+k+(k-2)(\alpha(G)-1)
\end{eqnarray*}
By the inequality (\ref{H}),
$\sum_{i \in I}d_H(x_i) \le |H|-1$.
Hence
$\sum_{ i \in I}d_G(x_i) \le |G|+\kappa(G)+(k-2)(\alpha(G) -1)-1$,
a contradiction.

\bigskip
\noindent\textbf{Case 2.2.} $h \ge k+1$.
\medskip

By Claims \ref{U1} and \ref{x0a-1}, the assumption of Case 2 and
the choice of $r$ and $v_{2}$,
we have $\bigcup_{i=2}^{p}V_{i} \subseteq U=N_{C}(x_{0})$.
Since $x_{0} \in V_{1} \cup S$ by Claim \ref{U1}, 
this implies that $x_{0} \in S$.

\begin{Claim} \label{XV1}
$|X \cap V_{1}| \le k-1$.
\end{Claim}

\begin{proof}
Suppose that
$|X \cap V_{1}| \ge k$.
Let $I$ be a subset of $M_{1}$
such that
$|I|=k$
and
$I \subseteq \{i : x_{i} \in X \cap V_{1}\}$.
Then
$\{x_i : i \in I\} \cup \{v_{2}\}$ is an independent set
of order $k+1$.
Let $s$ and $t$ be integers in $I$.
Since
$x_s,x_t \in V_1$,
$D \subseteq V_1 \cup S$
and
$\bigcup_{i=2}^{p}V_{i} \subseteq U$,
the similar argument as that of the inequality (\ref{2vertices})
implies that
\begin{equation*}
\dg{C}{x_s} +\dg{C}{x_t}
\le
|C \cap (V_{1} \cup S)|-\sum_{i \in I\sm \{s,t\}}|D_{i}|.
\end{equation*}
By the inequalities (\ref{C}) and (\ref{xwH}),
we have
$
\sum_{i \in I \sm \{s,t\}}\dg{C}{x_i}
\le  \sum_{i \in I\sm \{s,t\}}|D_{i}|+(k-2)(\alpha(G)-1)$
and
$
\sum_{i \in I }\dg{H}{x_i}
\le  |H \cap (V_{1} \cup S)|-1$,
respectively.
On the other hand,
we obtain
$\dg{G}{v_{2}} \le |V_{2} \cup S|-1.$
By these four inequalities,
$\sum_{i \in I}\dg{G}{x_i}+\dg{G}{v_{2}}
\le n+\kappa(G)+(k-2)(\alpha(G)-1)-2,$
a contradiction.
Therefore
$|X \cap V_{1}|  \le k-1$.
\end{proof}

Recall $U_{1}=\{u_i \in U : x_i \in X \cap V_{1}\}$.
By Claim \ref{XV1},
we have
$|U_1| \le k-1$.
By the assumption of Case 2.2 and the choice of $x_{1}$, 
we obtain
$A_2 \cap U_1= \emptyset$,
and hence
we can take
a subset $I$ of $\{2,3,\ldots,h\}$
such that $|I|=k$
and
$\{i : A_{i+1} \cap U_{1}\not=\emptyset\} \subseteq I$.
Let 
$$X_{I}=\{x_i : i \in I\}.$$
By Claim \ref{XW},
$X_{I} \cup \{x_{0}\}$ is an independent set
of order $k+1$.
Let
\begin{align*}
B_1 = B_{h+1} =  \ir{C}{u_1}{u_{h}}
\text{\ \ and\ \ } B_i =  \ir{C}{u_i}{u_{i-1}} \text{\ \ for } 2 \le i \le h.
\end{align*}
Then,
since
$|\Ir{C}{u_i}{u_{i}'}| \ge 2$ for $i \in M_1 \setminus I$,
the following inequality holds:
\begin{eqnarray*}
|C|
& \ge& \sum_{i \in I}|B_i\cup \{u_i\}|+2\Big(|U|-\sum_{i \in I}|(B_i\cup \{u_i\}) \cap U|\Big)\\
& =& \sum_{i \in I}|B_i|+2\Big(|U|-\sum_{i \in I}|B_i \cap U|\Big)-k.
\end{eqnarray*}
If $x_{i} \in X_{I} \cap S$,
then
it follows from Lemma \ref{insertible} (i)
and Claim \ref{c}
that 
\begin{eqnarray*}
d_C(x_i)
&\le&
 \Big(|U|- |B_{i}\cap U|- |B_{i+1} \cap U_{1}|\Big) + \Big(|B_i|-|\{x_{i}\}|-|(B_i \cap U)^{+}|\Big)\\
&=&
 |U|+ |B_i|-2|B_i \cap U|-|B_{i+1} \cap U_{1}|-1.
\end{eqnarray*}
If $x_{i} \in X_{I} \cap V_{1}$,
then,
by Lemma \ref{insertible} (i)
and Claim \ref{c},
\begin{eqnarray*}
d_C(x_i)
&\le&
 \Big(|U|- |B_{i}\cap U|- |B_{i+1} \cap U_{1} |-|(U \cap V_{2}) \sm B_{i}|+ |B_{i+1} \cap U_1 \cap V_2|\Big)\\
 &&{}+ \Big(|B_i|-|\{x_{i}\}|-|(B_i \cap U)^{+}|- |U \cap V_{2} \cap B_{i}|\Big)\\
&=&
 |U| + |B_i|-2|B_i \cap U|-|B_{i+1} \cap U_{1}|-1- \Big(|U \cap V_{2}| - |B_{i+1} \cap U_1 \cap V_2|\Big).
\end{eqnarray*}
Since $U \cap V_2 \neq \emptyset$, we obtain 
$|U \cap V_{2}| - |B_{i+1} \cap U_1 \cap V_2| \ge 1$ for all $i \in I$ except for at most one,
and hence
$$\sum_
{i \in I\,:\,x_i \in  X_I \cap V_1}
\Big(|U \cap V_{2}| - |B_{i+1} \cap U_1 \cap V_2|\Big) \ge |X_{I} \cap V_{1}|-1.$$

By the choice of $I$,
we have
$$|U_{1}|=\sum_{i \in I}|A_{i+1} \cap U_{1}|
= \sum_{i \in I}|B_{i+1} \cap U_{1}|+ \big|\{u_{i} : x_{i} \in X_{I} \cap V_{1}\}\big|.$$
On the other hand,
since $x_{0} \in S$,
it follows from Claim \ref{U1}
that
$$
|U_{1}|=|X \cap V_1|  = |X \sm S| \ge |X| -(\kappa(G)-1).$$
Moreover, by Claim \ref{x0a-1}, 
\begin{align*}
d_{C}(x_{0}) = |U| = |X| = \alpha(G) - 1. 
\end{align*}
Thus,
we deduce
\begin{eqnarray*}
\sum_{i \in I \cup \{ 0 \}}d_C(x_i)
&\le&
(k+1)|U|
+\sum_{i \in I}|B_i|
-2\sum_{i \in I}|B_i\cap U|
\\
&&{} 
-\sum_{i \in I}|B_{i+1} \cap U_{1}|-k-(|X_{I} \cap V_{1}|-1)
\\
&=&
\Big( \sum_{i \in I}|B_i|+2\big(|U|-\sum_{i \in I}|B_i \cap U|\big)-k \Big)
+
(k-1)|U|\\
&&{}-
\Big( \sum_{i \in I}|B_{i+1} \cap U_{1}|+ \big|\{u_{i} : x_{i} \in X_{I} \cap V_{1}\}\big| \Big) + 1\\
&\le&
|C|
+(k-1)|U|
+\kappa(G)-|X|\\
&=&
|C|
+\kappa(G)
+(k-2)(\alpha(G)-1).
\end{eqnarray*}
By the inequality (\ref{H}),
$\sum_{i \in I \cup \{ 0 \}}d_H(x_i) \le |H|-1$.
Hence
$\sum_{ i \in I \cup \{ 0 \}}d_G(x_i) \le |G|+\kappa(G)+(k-2)(\alpha(G) -1)-1$, a contradiction.
\end{proof}



\begin{thebibliography}{99}

\bibitem{Ainouche92}
A.~Ainouche,
An improvement of Fraisse's sufficient condition for hamiltonian graphs,
J.~Graph Theory \textbf{16} (1992), 529--543.


\bibitem {BBVL}
D.~Bauer, H.J.~Broersma, H.J.~Veldman and  R.~Li,
A generalization of a result of H\"aggkvist and Nicoghossian,
J.~Combin.~Theory~Ser.~B
\textbf{47}
(1989),
237--243.




\bibitem{Bondyrem} J.A.~Bondy,
A remark on two sufficient conditions for Hamilton cycles,
Discrete Math. \textbf{22} (1978), 191--193.

\bibitem{Bondy}
J.A.~Bondy,
Longest paths and cycles
in graphs with high degree,
Research Report CORR 80-16,
Department of Combinatorics and Optimization,
University of Waterloo,
Waterloo,
Ontario,
Canada
(1980).


\bibitem{Bondybook}
J.A.~Bondy,
``Basic Graph Theory: Paths and Circuits"
in: HANDBOOK OF COMBINATORICS,
Vol.~I,
eds. R.~Graham, M.~Gr\H{o}tshel and L.~Lov\'{a}sz
(Elsevier, Amsterdam),
1995,
pp.~5--110.


\bibitem {Chvatal&Erdos}
V.~Chv\'{a}tal and P.~Erd\H{o}s,
A note on hamiltonian circuits,
Discrete Math.
\textbf{2}
(1972),
111--113.

\bibitem{Dirac} G.A.~Dirac,
Some theorems on abstract graphs,
Proc.~London Math.~Soc.~\textbf{2} (1952), 69--81.

\bibitem {FJ} P.~Fraisse and H.~A.~Jung, 
``Longest cycles and independent sets in $k$-connected graphs,'' 
Recent Studies in Graph Theory, V.R.~Kulli,(Editor), 
Vischwa Internat. Publ. Gulbarga, India, 1989, pp. 114--139. 


\bibitem{HLT} A.~Harkat-Benhamadine, H.~Li and F.~Tian,
Cyclability of 3-connected graphs,
J.~Graph Theory \textbf{34} (2000), 191--203.


\bibitem{Hao13}
H.~Li,
Generalizations of Dirac's theorem in Hamiltonian graph theory\,--\,A survey,
Discrete Math.~\textbf{313} 
(2013), 2034--2053.

\bibitem{LTX10}
H.~Li, F.~Tian, Z.~Xu,
Hamiltonicity of $4$-connected graphs,
Acta Math.~Sin.~(Engl.~Ser.) \textbf{26} 
(2010), 699--710.


\bibitem{Ore}
O.~Ore,
Note on Hamilton circuits,
Amer. Math. Monthly
\textbf{67}
(1960),
55.

\bibitem {Ota}
K.~Ota,
Cycles through prescribed vertices with large degree sum,
Discrete Math.
\textbf{145}
(1995),
201--210.

\bibitem{Ozeki} K.~Ozeki,
Hamilton Cycles, Paths and Spanning Trees in a Graph,
Doctor thesis, Keio University (2009).



\bibitem{OY}
K.~Ozeki and T.~Yamashita,
A degree sum condition
concerning the connectivity and the independence number
of a graph,
Graphs Combin.~\textbf{24} (2008), 469--483.


\end{thebibliography}
\end{document}